\documentclass[11pt]{amsart}
\usepackage[margin=3.33cm]{geometry}
\usepackage{mathscinet,mathrsfs,amsmath,amsfonts,amssymb,amsthm}
\usepackage{bbm}
\usepackage{graphicx}
\usepackage[usenames,dvipsnames,svgnames,table]{xcolor}
\usepackage[linktocpage=true,colorlinks=true,linkcolor=RubineRed!85!black,citecolor=ForestGreen,urlcolor=green,pdfborder={0 0 0}]{hyperref}
\usepackage[safe]{tipa}
\usepackage{upgreek}

\theoremstyle{plain}
\newtheorem{theorem}{Theorem}[section]
\newtheorem{corollary}[theorem]{Corollary}
\newtheorem{proposition}[theorem]{Proposition}
\newtheorem{lemma}[theorem]{Lemma}
\theoremstyle{definition}
\newtheorem{definition}[theorem]{Definition}
\newtheorem{remark}[theorem]{Remark}
\numberwithin{equation}{section}

\newcommand*{\dif}{\mathop{}\!\mathrm{d}}
\def \supp {{\text{\rm supp}} \hspace{.5mm}}
\def \div {{\text{\rm div}} \hspace{.5mm}}
\def \diag {{\text{\rm diag}}}
\def \trace {{\text{\rm tr}}}
\def \det {{\text{\rm det}}}
\def \osc {{\text{\rm osc}}}
\def \dist {{\hspace{.3mm}\text{\rm dist}}}
\def \fin {f_{\rm in}}
\def \Rm {\mathcal{R}_m}
\def \p {\partial}
\def \N {\mathbb{N}}
\def \R {\mathbb{R}}
\def \I {\mathcal{I}}
\def \M {{\mathcal{M}}}
\def \NN {{\mathcal{N}\hspace{-.3mm}}}
\def \Gi {\mathcal{G}}
\def \Ge {\Gamma_{\;\!\!\;\!\!e}}
\def \BB {\mathscr{B}}
\def \T {{\mathcal{T}}}
\def \Q {\mathcal{Q}}
\def \SS {\mathcal{S}}
\def \RR {\mathcal{R}}
\def \V {\mathcal{V}}
\def \W {\mathcal{W}}
\def \U {\mathcal{U}}
\def \Y {\mathcal{Y}}
\def \Z {\mathcal{Z}}
\def \D {\mathcal{D}}
\def \OO {\mathcal{O}}

\def \mm {{\boldsymbol m}}
\def \nn {{\boldsymbol n}}
\newcommand\GG{\reflectbox{\rotatebox[origin=c]{180}{$\mathbb L$}}}
\newcommand{\la}{\langle}
\newcommand{\ra}{\rangle}
\newcommand{\lv}{\langle v\rangle}
\newcommand{\lvv}{\langle v_0\rangle}

\begin{document}
\title[Regularity of kinetic Fokker-Planck equations in bounded domains]{Regularity of kinetic Fokker-Planck equations in bounded domains}	
\date{\today}
\author{Yuzhe ZHU}
\address{D\'epartement de math\'ematiques et applications, \'Ecole normale sup\'erieure\;-\,PSL, 45 Rue d'Ulm, 75005 Paris, France}
\email{yuzhe.zhu@ens.fr}
\curraddr{Department of Mathematics, University of Chicago, 5734 S University Ave, Chicago, Illinois 60637, USA}
\email{yuzhezhu@uchicago.edu}


\begin{abstract}
We obtain the existence, uniqueness and regularity results for solutions to kinetic Fokker-Planck equations with bounded measurable coefficients in the presence of boundary conditions, including the inflow, diffuse reflection and specular reflection cases. 
\end{abstract}

\maketitle

\hypersetup{bookmarksdepth=2}
\setcounter{tocdepth}{1}
\tableofcontents

\section{Introduction} 
The main purpose of this work is to investigate the a priori regularity estimates for solutions to kinetic Fokker-Planck equations when supplemented with one of the following boundary conditions: inflow injection, diffuse reflection, and specular reflection. Of concern is the equation of the hypoelliptic form 
\begin{align}\label{FP}
(\p_t+v\cdot\nabla_x)f
=\nabla_v\cdot(A\nabla_v f)+B\cdot\nabla_vf +cf + s, 
\end{align}
for an unknown function $f=f(z)$ with $z:=(t,x,v)\in(0,T)\times\Omega\times\R^d$, where $T>0$, $\Omega$ is a bounded domain in $\R^d$, and the $d\times d$ real symmetric matrix $A=A(z)$, the $d$-dimensional vector $B=B(z)$ and the scalar functions $c=c(z),s=s(z)$ are given. We are always under the assumption that there is some constant $\Lambda>1$ such that in $(0,T)\times\Omega\times\R^d$, 
\begin{align}\label{H}
\left\{ 
\begin{aligned}
\ &\,\Lambda^{-1}|\xi|^2\le A\,\xi\cdot\xi\le\Lambda|\xi|^2
{\quad\rm for\ any\ }\xi\in\R^d, \\
\ &|B|+|c|\le\Lambda. \\
\end{aligned}
\right. 
\end{align}

\subsection{Main results}
Before stating the results, let us first make the notion of boundary conditions precise. 

\subsubsection{Phase boundaries} 
For $x\in\p\Omega$, the unit outward normal vector of $\p\Omega$ at $x$ is denoted by $n_x$. Let $\OO:=\Omega\times\R^d$ denote the phase domain. We split the phase boundary $\Gamma:=\p\Omega\times\R^d$ into the outgoing part $\Gamma_+$, incoming part $\Gamma_-$, and grazing (characteristic) part $\Gamma_0$, which are defined by 
\begin{align*}
\Gamma_\pm:=& \big\{(x,v)\in\p\Omega\times\R^d : \pm\, n_x\cdot v >0\big\},\\
\Gamma_0:=& \big\{(x,v)\in\p\Omega\times\R^d : n_x\cdot v =0\big\}.
\end{align*}
We denote by $\gamma f$ the trace of a (weak) solution $f:\overline{\OO}\rightarrow\R$ to \eqref{FP} at the boundary set $\Gamma$, and $\gamma_\pm f:=\gamma f\mathbbm{1}_{\Gamma_\pm}$. It will be shown in Section~\ref{EC} that our notion of solutions always admits suitable trace at boundaries. For $t\in(0,T]$, we abbreviate 
\begin{align*}
\Sigma_t^\pm&:=[0,t]\times\Gamma_\pm, \\
\Sigma_t&:=[0,t]\times\Gamma, \\ 
\OO_t&:=(0,t)\times\OO. 
\end{align*}

\subsubsection{Inflow boundary condition}
Given a function $g:\Sigma_T^-\rightarrow\R$ as the boundary data, we impose the condition, with the notation of the single-valued operator $\Gi$ for convenience, 
\begin{align*}
f(t,x,v)=\Gi f:=g(t,x,v)  {\quad\rm in\ }\Sigma_T^-. 
\end{align*}
In particular, the case $g=0$ corresponds to the so-called \emph{absorbing} boundary condition. 

\subsubsection{Nonlocal reflection boundary condition} 
The nonlocal boundary condition of a generalized diffuse reflection type to be concerned with is as follows,  
\begin{align*}
f(t,x,v)=\NN f:=\M(t,x,v)\int_{\R^d}f(t,x,v')\;\!(n_x\cdot v')_+\dif v'  {\quad\rm in\ }\Sigma_T^-,
\end{align*}
where the function $\M\in C^\beta(\Sigma_T)$ with $\beta\in(0,1)$ is given, and satisfies that for any $q\ge0$, there exists some constant $\Lambda_q>0$ such that
\begin{align}\label{M}
\|\lv^q\M\|_{L^\infty(\Sigma_T)} + [\M]_{C^\beta(\Sigma_T)}\le\Lambda_q.  
\end{align} 
Here the bracket $\la\cdot\ra:=(1+|\cdot|^2)^{1/2}$, and $C^\beta$ with $\beta\in(0,1)$ denotes the classical H\"older space with exponent $\beta$. When $\M$ has a form of the boundary Maxwellian, that is $\M=(2\pi)^{-\frac{d-1}{2}} \Theta^{-\frac{d+1}{2}} e^{-\frac{|v|^2}{2\Theta}}$ for some (uniformly positive and bounded) boundary temperature function  $\Theta:[0,T]\times\p\Omega\rightarrow\R_+$, the  operator $\NN$ is called \emph{diffuse reflection}. 

\subsubsection{Specular reflection boundary condition} 
The boundary condition with respect to the specular reflection operator $\RR$ reads  
\begin{align*}
f(t,x,v)=\RR f:=f(t,x,v-2(n_x\cdot v)\;\!n_x)  {\quad\rm in\ }\Sigma_T^-.  
\end{align*}

\subsubsection{Statement of the main theorem} 
Our results provide well-posedness and H\"older a priori bounds for solutions of \eqref{FP} supplemented with one of the above three boundary conditions. 

\begin{theorem}\label{thm}
Let the domain $\Omega$ be bounded with $\p\Omega\in C^{1,1}$, and let $\BB\in\{\Gi,\NN,\RR\}$ be the boundary operator. 
Assume that \eqref{H} and \eqref{M} hold. \vspace{-2.5mm}
\begin{item}
\item[\ \,$\bullet$ (Well-posedness)] 
For any $m\ge0$, we have some constant $l>0$ depending only on $d,m$ such that, for any given functions $\fin,s,g$ satisfying $\lv^l\fin\in L^\infty(\OO)$, $\lv^ls,\lv^lg\in L^\infty(\OO_T)$, there exists a unique bounded weak solution $f$ to \eqref{FP} such that $f|_{t=0}=\fin$ and $\gamma_-f=\BB f$, and such that for some constant $C>0$ depending only on $d,T,\Lambda,m,\Omega,\M$, we have 
\begin{align*}
\|\lv^mf\|_{L^\infty(\OO_T)}
\le C\big(\|\lv^ls\|_{L^\infty(\OO_T)} +\|\lv^l\fin\|_{L^\infty(\OO)} +\mathcal{B}\big), 
\end{align*}
where $\mathcal{B}=\|\lv^lg\|_{L^\infty(\Sigma_T^-)}$ when $\BB=\Gi$, and $\mathcal{B}=0$ when $\BB\in\{\NN,\RR\}$. 

\item[\ \,$\bullet$ (H\"older regularity)] 
If additionally $\fin\in C^\beta(\OO)$ and $g\in C^\beta(\Sigma_T^-)$ with $\beta\in(0,1]$, and the compatibility condition $\gamma_-\fin=\gamma_-\BB\fin$ holds, then there are some constants $\alpha\in(0,1)$ and $C'>0$ depending only on $d,T,\Lambda,m,\beta,\Omega,\M$ such that 
\begin{align*}
\|\lv^mf\|_{L^\infty(\OO_T)}+[f]_{C^\alpha(\OO_T)}\\
\le C'\big(\|\lv^ls\|_{L^\infty(\OO_T)} +\|\lv^l\fin\|_{L^\infty(\OO)}+[\fin]_{C^\beta(\OO)} +\mathcal{B}' \big)&, 
\end{align*}
where $\mathcal{B}'=\|\lv^lg\|_{L^\infty(\Sigma_T^-)}+[g]_{C^\beta(\Sigma_T^-)}$ when $\BB=\Gi$, and $\mathcal{B}'=0$ when $\BB\in\{\NN,\RR\}$.
\end{item} 
\end{theorem}

\begin{remark}
The theorem consists in Propositions~\ref{global}, \ref{globalr}, \ref{globals}. The estimates can be localized; see Propositions~\ref{local}, \ref{localr}, \ref{locals}. More precisely, under the same assumption as in the above theorem, we have the following local-in-time estimates written in a unified way with the abbreviation $\OO_t^\tau:=[\max\{0,\tau-t\},\tau]\times\overline{\OO}$ for $\tau,t\in(0,T]$. For any $m>0$, there exist some constants $l,C>0$ such that for any $\tau\in(0,T]$, we have 
\begin{align*}
\|\lv^mf\|_{L^\infty(\OO_1^\tau)}
\le C \big(\|\lv^{l} f\|_{L^2(\OO_2^\tau)}+ \|\lv^{l}s\|_{L^\infty(\OO_2^\tau)}+ \|\lv^{l}f\|_{L^\infty(\{t=0\}\cap\OO_2^\tau)} +\mathcal{B}_\tau \big), 
\end{align*}
where $\mathcal{B}_\tau=\|\lv^lg\|_{L^\infty(\Sigma_T^-\cap\OO_2^\tau)}$ when $\BB=\Gi$, and $\mathcal{B}_\tau=0$ when $\BB\in\{\NN,\RR\}$; and 
\begin{align*}
\|\lv^mf\|_{L^\infty(\OO_1^\tau)} + [f]_{C^\alpha(\OO_1^\tau)}\\
\le C\big(\|\lv^{l} f\|_{L^2(\OO_2^\tau)}+ \|\lv^ls\|_{L^\infty(\OO_2^\tau)} +[f]_{C^\beta(\{t=0\}\cap\OO_2^\tau)} +\mathcal{B}_\tau' \big)&, 
\end{align*} 
where $\mathcal{B}_\tau'=\|\lv^lg\|_{L^\infty(\Sigma_T^-\cap\OO_2^\tau)}+[g]_{C^\beta(\Sigma_T^-\cap\OO_2^\tau)}$ when $\BB=\Gi$, and $\mathcal{B}_\tau'=0$ when $\BB\in\{\NN,\RR\}$. 
\end{remark}

\subsection{Backgrounds and related work}
\subsubsection{Kinetic boundary value problems}
As a formulation of stochastic processes, the equation~\eqref{FP} appears naturally in the Langevin theory of Brownian motion; see the review \cite{Chand}. It describes the system constituted of a large number of interacting particles in the phase space, arising from the study of plasma physics and galactic dynamics for instance. The solution $f(t,x,v)$ to \eqref{FP} is interpreted as the density evolution of particles at time $t$ occupying the phase state of position $x$ and velocity $v$. 

When the interacting particles are confined in a bounded domain, the equation has to be supplemented with physically relevant boundary conditions that take into account how particles behave at the boundary; see \cite{Maxwell}, \cite{Cerci2}. As its name suggests, the inflow boundary condition means that the density of the particles flowing inward the domain is prescribed. We can see from this viewpoint that boundary conditions are free from prescriptions whenever the particles exit from the boundary. The reflection boundary conditions take the form of balance relations between the densities at the incoming and outgoing boundaries. The diffuse reflection as a nonlocal model describes that the striking particles are thermalized and then re-emitted inside the domain according to the boundary state. The interaction of particles with perfect solid boundaries is modeled by the specular reflection, meaning that particles are re-emitted elastically with postcollisional angles equal to the precollisional angles, as if light rays are reflected by a perfect mirror in optics. Despite the importance of the treatment of boundary conditions in the study of boundary value problems, limited results on boundary regularity for \eqref{FP} are known. 

The assumption~\eqref{H} without requirement of regularity is a bridge to the nonlinearity of various collisional kinetic models such as the Landau equation proposed in \cite{Landau}. Combining Theorem~\ref{thm} with the result of interior estimates in \cite{HS}, we know that any bounded positive solution to the Landau equation is smooth inside and H\"older continuous up to the boundary in a general bounded domain, where its boundary regularity cannot be improved in some sense as we will discuss in \S\ref{chara} below. In contrast, the Boltzmann equation with angular cutoff has different regularizing effects; its regularity property in convex domains was investigated in \cite{GKTT}, and discontinuities of the solutions were shown to be created at nonconvex parts of the boundary and propagate inside the domain along characteristics in \cite{Kim}. 

\subsubsection{Trace problem in kinetic equations}\label{trace-pro}
It is of importance to make sense for the trace of solutions when dealing with boundary value problems. In the kinetic setting, the transport part $\p_t+v\cdot\nabla_x$ of equations is hyperbolic and therefore lacks regularity. Even if exploiting the velocity diffusion in \eqref{FP}, fully characterizing the traces remains challenging, especially within the functional spaces where the solutions of \eqref{FP} exist, except in the one-dimensional case ($d=1$) as demonstrated in \cite[Theorem~1]{BG}. Variational frameworks for \eqref{FP} in a general bounded domain have not been well-developed; see \cite{AM2021} for some discussions and references. Nevertheless, the equation itself can be additionally used to circumvent this tricky issue. 

Some progress has been made since the earlier fundamental work on the trace problem, such as \cite{Bardos}, \cite{Cessenat2} for transport equations, and \cite{Ham} for the Boltzmann equation. Through regularization method based on the theory of renormalized solutions developed in \cite{DPLode}, general treatments for the traces with the aid of Green's renormalization formula for solutions to a large class of kinetic equations were studied in \cite{Mischler1}, \cite{Mischler3}. 

We will exploit the regularization method related to a renormalization technique, in combination with the classical energy method for parabolic equations, to develop the weak theory of initial-boundary value problems for \eqref{FP} under the $L^2$-framework in Section~\ref{EC}. The renormalization formula established in Subsection~\ref{RU} plays a role not only in deducing the uniqueness of solutions to \eqref{FP}, but also in shedding light on the trace for general bounded solutions. In Subsection~\ref{EW}, we construct solutions to the inflow boundary problems by solving a sequence of approximating parabolic equations, and then derive solutions to the specular reflection boundary problems through an iterative scheme of inflow problems. 

\subsubsection{Hypoellipticity}
Let us discuss here (and also in \S\ref{chara}) the equation~\eqref{FP} with smooth coefficients in $(0,T)\times\Omega\times\V$, for some open sets $\Omega,\V\subset\R^d$. The main part of \eqref{FP} subject to \eqref{H} can be written as H\"ormander's summation form (of type II)
\begin{align*}
\mathscr{L}:=X_0+\sum\nolimits_{i=1}^d X_i^*X_i=\p_t+v\cdot\nabla_x-\nabla_v\cdot(A\nabla_v\cdot), 
\end{align*}
where the vector fields $X_0:=\p_t+v\cdot\nabla_x$ and $(X_1,X_2,...,X_d)^T:=\sqrt{A}\nabla_v$ with the formal adjoint $X_i^*$ for $X_i$. As a basic observation, the commutator $[\nabla_v,X_0]=\nabla_x$. The notion of \emph{hypoellipticity} refers that the Lie algebra generated by the system $\{X_i\}_{i=0}^d$ span the full tangent space. An anisotropic diffusion in the operator $\mathscr{L}$ was first noticed by Kolmogorov in \cite{Kol} through the explicit calculation of its fundamental solution. It was then shown by H\"ormander in \cite{Ho} that the hypoelliptic structure of $\mathscr{L}$ ensures a (interior) regularization effect that the smoothness of $\mathscr{L}f$ implies the smoothness of $f$. 

\subsubsection{Characteristic points}\label{chara}
It is a classical difficulty when concerned with the regularity up to the boundary for solutions to degenerate elliptic equations, especially due to the presence of characteristic points. A boundary point associated to $\mathscr{L}$ is called \emph{characteristic} if every vector field $X_i$ with $0\le i\le d$ is tangent to the boundary at this point. The boundary points of the phase domain $\Omega\times\V$ is thus classified as $\Omega\times\p\V$, $\{(x,v)\in\p\Omega\times\V:\pm n_x\cdot v>0\}$, and the characteristic portion $\{(x,v)\in\p\Omega\times\V:n_x\cdot v=0\}$. When $\V=\R^d$, characteristic set coincides with $\Gamma_0$. 

As a historical remark, on the one hand, there have been several results on the issue of boundary regularity for degenerate elliptic equations since \cite{Keld}, \cite{Fichera1}, where the phenomena of loss of boundary conditions resulting from the degeneracy was noticed. The study of boundary regularity near certain non-characteristic points in a general setting can be found in \cite{KN1}, \cite{Oleinik2}. On the other hand, it was discovered that the loss of regularity occurs for some particular hypoelliptic problem (of type I)  at characteristic points. Indeed, an explicit solution to the Dirichlet problem associated with the Kohn Laplacian acting on the Heisenberg group was constructed in \cite{Jerison}, which was shown to be vanishing on the boundary and not better than H\"older continuous. However, none of the above results address full boundary issues for $\mathscr{L}$. 

Under the assumption of a simple structure with constant coefficients for the operator $\mathscr{L}$, the concern of continuity for solutions associated with the absorbing boundary condition was analyzed in \cite{HJV}, based on the barrier argument and the study of self-similar behaviors of solutions. Despite the lack of a complete description in \cite{HJV}, by following their argument, one is able to obtain a solution to \eqref{FP} equipped with constant coefficients and the absorbing boundary condition, that is at most in a H\"older class near the characteristic point. Its elaboration is presented in Appendix~\ref{counterexample}. 

\subsubsection{De Giorgi's technique}\label{dg}
One of the main parts of the proof for the regularity result in Section~\ref{inflow} relies primarily on the technique pioneered by De Giorgi in \cite{DG} for elliptic equations with bounded measurable coefficients. Its basic idea is to build up an oscillation decay of solutions at the unit scale; as a consequence of the scaling-translation invariance of equations, the oscillation control holds at each scale and hence yields interior H\"older estimates for solutions. Based on its hypoelliptic nature, the counterpart of regularization effect for $\mathscr{L}$ with rough coefficients was first obtained in \cite{Zhang}. An alternative approach with more comprehensive descriptions of properties for subsolutions to \eqref{FP} was proposed in \cite{GIMV}. 

For the elliptic case, the boundary estimate of solutions depends on certain geometric condition on domains. The interior regularity result extended to boundary for domains satisfying the exterior cone condition is well-known; see for instance \cite{GT}. Indeed, the combination of the cone condition and the H\"older continuous boundary data implies the same oscillation decay at each scale as in the interior case. In a similar manner, for \eqref{FP}, on the portions of boundary where the values are specified, that is $\Gamma_-$, one can deduce the boundary estimate from the interior regularization mechanism. 

\subsubsection{Extension method}
The regularity of inflow boundary problems for \eqref{FP} is fully treated in Section~\ref{inflow}. To overcome the difficulty of regularity due to the loss of boundary conditions on $\Gamma_+$ and the characteristic set $\Gamma_0$, we develop an extension method that reduces the singular boundary problems to a manageable scope. The proper extension across the boundary portion $\Gamma_+\cup\Gamma_0$ is guaranteed by the existence result of general inflow problems established in Section~\ref{EC}. This will lead to modified boundary value problems with fully specified boundary conditions. Provided that the inflow data is H\"older continuous, the treatment of oscillation controls on the boundary yields H\"older estimates for the extended solutions. 

The formulation of the nonlocal reflection boundary condition is essentially the same as the one of inflow problems. Although the incoming data is self-induced, we will show in Section~\ref{diffuse} that the macroscopic boundary quantity $\NN f$ with the solution $f$ to \eqref{FP} is actually H\"older continuous. The regularity estimate in such nonlocal reflection boundary problems is then obtained from the result of the inflow injection case in Section~\ref{inflow}. Moreover, the boundary a priori estimate is also used in Section~\ref{diffuse} to show the existence and uniqueness for this kind of reflection boundary problems. 

As for specular reflection boundary problems, the boundary regularity for solutions is proved in Section~\ref{specular} through a mirror extension method studied in \cite{Nier}, \cite{GHJO} with the aid of the trace result obtained in Section~\ref{EC}. There have been some development in certain special cases of \eqref{FP} in \cite{GHJO}, \cite{DGY}. Their key observation is that the solution can be extended through this extension trick outside of the domain continuously even near the characteristic set, which gives a direct reduction to interior issues. A similar mirror extension trick has been widely used in the Neumann problem for elliptic equations; see for instance \cite[Section~2.4.3]{Tro}, \cite[Proof of Theorem~2.5]{DK}, \cite[Section~9.1.8]{Nier}. 

\subsection{Notations} 
\subsubsection{Boundary conventions} 
We recall that the phase domain $\OO=\Omega\times\R^d$ and the phase boundary $\Gamma=\p\Omega\times\R^d=\Gamma_+\cup\Gamma_0\cup\Gamma_-$. Let the time-space domain $\Omega_T:=[0,T]\times\Omega$, and $n_{t,x}$ be the unit outward normal vector of $\p\Omega_T$, and $\dif\sigma_{t,x}$ be the surface measure on $\p\Omega_T$. We define the measure $\dif\mu$ on $\p\OO_T$ by 
\begin{align*}
\dif\mu:=|n_{t,x}\cdot (1,v)|\dif\sigma_{t,x}\dif v.
\end{align*}
We abbreviate the effective initial-boundary portion $\Ge$ of the domain $(0,T)\times\Omega\times\R^d$ by 
\begin{align*}
\Ge:&=\big(\{0\}\times\overline{\Omega}\times\R^d\big)
\cup\left([0,T]\times\Gamma_-\right)\\ &=\left(\{0\}\times\overline{\OO}\right)\cup\Sigma_T^-, 
\end{align*}
where the initial-boundary condition for \eqref{FP} can only be specified. 

\subsubsection{Invariant transformation}\label{invariant}
For $z_0=(t_0,x_0,v_0)\in\R^{1+2d}$ and $r>0$, we define the transformation $\T_{z_0,r}:\R^{1+2d}\rightarrow\R^{1+2d}$ by the prescription: 
\begin{align*}
\T_{z_0,r}:\, (\tilde{t},\tilde{x},\tilde{v}) \longmapsto
(t,x,v):=(t_0+r^2\tilde{t},\,x_0+r^3\tilde{x}+r^2\tilde{t}v_0,\,v_0+r\tilde{v}). 
\end{align*}
We abbreviate the cylinder centered at the origin of radius $r>0$ as $Q_r:= (-r^2, 0] \times B_{r^3}(0) \times B_r(0)$. The general cylinder centered at $z_0$ with radius $r$ is defined by 
\begin{align*}
Q_r(z_{0}):=&\left\{ \T_{z_0,r}(\tilde{z}): \, \tilde{z}\in Q_1\right\}\\
=&\left\{ (t,x,v): \, t_{0} - r^{2} < t \le t_{0}, \, |x - x_{0} - (t - t_{0})v_{0}| < r^{3},\, |v - v_{0}| < r\right\}.
\end{align*}
Loosely speaking, \eqref{FP} is invariant under the transformation, as the composition $f\circ\T_{z_0,r}$ of a solution $f$ to \eqref{FP} in $Q_r(z_{0})$ will solve an equation with the same structure in $Q_1$. 

As a technical remark, \eqref{FP} is not really translation invariant around the boundary, since the transformation $\T_{z_0,r}$ typically does not preserve the formulation of boundary conditions. Consequently, we avoid using the normal kinetic H\"older spaces, as it always necessitates more careful attention to the quantitative analysis involving velocity weights; see Section~\ref{inflow}. 

\subsubsection{Other notations}\label{notations}
Throughout the article, $B_R(\zeta)$ with $R>0$, $\zeta\in\R^N$ and $N\in\N_+$ denotes the Euclidean ball in $\R^N$ centered at $\zeta$ with radius $R>0$. 

Let $e_d\in\R^d$ be the $d$-th coordinate vector. 

We write the positive part $a_+:=\max\{a,0\}$ and negative part $a_-:=\max\{-a,0\}$ for any real-valued function $a$. 

We define the weighted $L^p$ space, denoted by $L^p(U,w)$, as the set of functions $u$ for which $\||u|^pw\|_{L^1(U)}<\infty$, where $U$ is a measurable set and $w$ is a weight function. 

A constant $C$ is called \emph{universal} if it depends only on $d,T,\Lambda,m,\beta,\Omega,p,q,m,l,\omega,\epsilon$ specified in context. The symbol $X\lesssim Y$ designates $X\le CY$ for some universal constant $C>0$, and the symbol $X\approx Y$ means that $X\lesssim Y$ and $Y\lesssim X$. 

\subsection{Organization of the paper}
The outline of the article is as follows. In Section~\ref{EC}, we study the well-posedness of weak solutions to the equation with the inflow and specular reflection boundary conditions. Section~\ref{inflow}, \ref{diffuse}, \ref{specular} are devoted to the study of the regularity issues in the inflow, nonlocal reflection and specular reflection boundary problems, respectively. The well-posedness result with the nonlocal reflection boundary condition is also derived in Section~\ref{diffuse}. We finally present in Appendix~\ref{counterexample} an example showing that the H\"older class is optimal near the boundary even for classical solutions. 

\medskip\noindent\textbf{Acknowledgement.}
The author is indebted to Fran\c{c}ois Golse, Cyril Imbert, and St\'{e}phane Mischler for helpful discussions. This work has received funding from the European Union’s Horizon 2020 research and innovation programme under the Marie Skłodowska-Curie grant agreement No 754362. 

\section{Theory of weak solutions}\label{EC}
This section is devoted to the theory of weak solutions to initial-boundary problems of the kinetic Fokker-Planck type equation in bounded domains. We find that neither the results nor the arguments known in the literature are complete. 

Let $T>0$ and $\D$ be a domain in $\R^d_x\times\R^d_v$. Assume that the boundary $\p\D$ is $C^{0,1}$ and consists only of the portions: the boundary $\p_x\D$ with respect to $x$ and the boundary $\p_v\D$ with respect to $v$. The boundary $\p_v\D$ can be empty, for instance $\D=\OO=\Omega\times\R^d$; and if $\D=\Omega\times\V$ for $\Omega,\V\subset\R^d$, then $\p_x\D=\p\Omega\times\V$ and $\p_v\D=\Omega\times\p\V$. We abbreviate $\D_t:=(0,t)\times\D$ for $t\in(0,T]$. The component of the unit outward normal vector at $x\in\p_x\D$ in $\R_x^d$ is denoted by $n_x$. We recall that the initial-boundary condition can only be specified on the effective boundary portion $\GG$ of $\D_T$, which is defined by 
\begin{align*}
\GG:=\left(\{0\}\times\overline{\D}\right)
\cup\big((0,T)\times\{(x,v)\in\p_x\D:n_x\cdot v<0\}\big)
\cup\big((0,T)\times\p_v\D\big). 
\end{align*} 
In this section, we consider a larger class of equations 
\begin{align}\label{existence}
(\p_t+v\cdot\nabla_x)f
=\nabla_v\cdot\left(A\nabla_v f\right) +B\cdot\nabla_vf +cf +\nabla_v\cdot G_1 +G_0 {\ \;\rm in\ }\D_T, 
\end{align}
where $G_1,G_0\in L^2(\D_T)$ are given, and the measurable coefficients $A,B,c$ satisfy the condition~\eqref{H} in $\D_T$. 

Let us now make the notion of weak solutions precise. 
\begin{definition}\label{def-weak}
A pair of functions 
\begin{align*}
(f,\gamma_xf)\in C^0([0,T];L^2(\D))\times L_{loc}^2([0,T]\times\p_x\D,|n_x\cdot v|)
\end{align*} 
is said to be a \emph{weak solution} to \eqref{existence} in $\D_T$, if $\nabla_vf\in L^2(\D_T)$, and for any $t\in(0,T]$ and $\varphi\in C^1_c(\overline{\D_t})$ with $\varphi=0$ on $[0,t]\times\p_v\D$, we have 
\begin{align}\label{weak}
\begin{aligned}
\int_{\{t\}\times\D} f\varphi  -\int_{\{0\}\times\D}f\varphi  
+ \int_{[0,t]\times\p_x\D}  (n_x\cdot v)\gamma_xf\varphi 
-\int_{\D_t} f(\p_t+v\cdot\nabla_x)\varphi \\
=\int_{\D_t} (-A\nabla_vf\cdot\nabla_v\varphi 
+ \varphi B\cdot\nabla_vf + c f\varphi -G_1\cdot\nabla_v\varphi +G_0\varphi)&. 
\end{aligned}
\end{align}
\end{definition}

\begin{remark}
For a given $f$ in the above definition, the function $\gamma_xf$ satisfying \eqref{weak} is unique, which is named as the trace of $f$ on $[0,T]\times\p_x\D$. Because of this uniqueness, we will sometimes refer to $f$ as the weak solution to \eqref{existence} below for simplicity, which actually indicates the pair of functions $(f,\gamma_xf)$. 

Since the notation $\gamma f$ denotes the trace of $f$ at the phase boundary $\Gamma=\p\Omega\times\R^d$, we use $\gamma_x f$ to emphasize that it is the trace of $f$ at the boundary $\p_x\D$ with respect to $x$ throughout this section. 
\end{remark}

\begin{remark}
As discussed in \S\ref{trace-pro}, due to the lack of development of bounded trace operators from $\big\{f\in L_{t,x}^2(([0,T]\times\Omega;H_v^1(\R^d)):(\p_t+v\cdot\nabla_x)f\in L_{t,x}^2([0,T]\times\Omega;H_v^{-1}(\R^d))\big\}$ to $\big\{f\in L_{loc}^2([0,T]\times\p\Omega\times\R^d,|n_x\cdot v|)\big\}$, we always have to make our notion of weak solutions stronger, which involves a pair of functions $(f,\gamma f)$. 
\end{remark}

\begin{remark}
In the following Sections \ref{inflow}, \ref{diffuse}, \ref{specular}, we will demonstrate that the solution $f$ to \eqref{FP} is (H\"older) continuous up to the boundary of the domain, contingent upon suitable assumptions, as shown in Theorem~\ref{thm}. Given the uniqueness of the trace, the trivial restriction of the continuous function $f$ on the boundary is naturally consistent with its trace $\gamma_x f$ defined above. 
\end{remark}

\subsection{Renormalization formula and uniqueness of weak solutions}\label{RU}
To establish the uniqueness of weak solutions to \eqref{existence} by energy estimates combined with a Gr\"onwall-type argument, we have to approximate the weak solution not only in $\D_t$ but also on its boundary $\p\D_t$. The argument of regularizing approximation stemmed from a renormalization technique is patterned after that of \cite{Mischler3}. In this subsection, we deal with the case that the phase domain $\D$ is a product space of domains $\Omega,\V\subset\R^d$ so that $\p_x\D=\p\Omega\times\V$ and $\p_v\D=\Omega\times\p\V$. 

\begin{lemma}[Renormalization formula]\label{trace}
Let $(f,\gamma_xf)$ be a weak solution to \eqref{existence} in $\D_T$ with $\D=\Omega\times\V$, for the domains $\Omega,\V\subset\R^d$ and the boundaries $\p\Omega\in C^{0,1},\p\V\in C^{0,1}$. 
Then, for any $\chi\in C^{1,1}(\R)$ such that $\chi'(f)=0$ on $[0,T]\times\Omega\times\p\V$ and $\chi(\iota)=O(\iota^2)$, $\chi''(\iota)=O(1)$ as $|\iota|\rightarrow\infty$, and any $\varphi\in C^1_c(\overline{\D_t})$ with $t\in(0,T]$, we have 
\begin{align}\label{ren}
\begin{aligned}
\int_{\{t\}\times\D} \chi(f)\varphi  -\int_{\{0\}\times\D}\chi(f)\varphi  
+ \int_{[0,t]\times\p\Omega\times\V}  (n_x\cdot v) \chi(\gamma_xf)\varphi&\\
=\int_{\D_t} \big[\chi(f)(\p_t+v\cdot\nabla_x)\varphi 
-A\nabla_v\chi(f) \cdot\nabla_v\varphi - \varphi\chi''(f)( A\nabla_vf+G_1)\cdot\nabla_vf&\\
+ \varphi B\cdot\nabla_v\chi(f) + c f\chi'(f)\varphi -\chi'(f) G_1\cdot\nabla_v\varphi + \chi'(f)G_0\varphi\big]&. 
\end{aligned}
\end{align}
Besides, provided that the weak solution $f$ is bounded in $\D_T$, the trace $\gamma_xf$ satisfies 
\begin{align*}
\|\gamma_xf\|_{L^\infty([0,T]\times\p_x\D)}\le\|f\|_{L^\infty(\D_T)}. 
\end{align*} 
\end{lemma}

\begin{proof}
We show the formula \eqref{ren} by an approximation argument only in the space variable for the sake of simplicity, as the approximation in other variables is standard. Let $\{\rho_k\}_{k\in\N_+}\subset C^\infty_c(\R^d)$ be a mollifier sequence such that 
\begin{align*}
\rho_1\ge 0,\quad 
\supp\rho_1\subset B_1,\quad 
\int_{\R^d}\rho_1(x)\dif x=1 \quad {\rm and}\quad 
\rho_k(x)=k^d\rho_1(kx).
\end{align*}
We take the domain $\Omega_\delta\subset\R^d$ with $C^{1,1}$-boundary, where the (small) parameter $\delta>0$ is intended to regularize and approximate $\Omega$ as $\delta\rightarrow0$. Define the unit vector $n_x^\delta:=|\nabla_x \dist(x,\p\Omega_\delta)|^{-1}\nabla_x \dist(x,\p\Omega_\delta)$. 

For a function $h\in L^2(\Omega)$, we define the convolution-translation regularization, 
\begin{align}\label{convolution}
h_{\star k}(y):=\int_{\Omega}h(x) \rho_k\big(y-2k^{-1} n_y^\delta-x\big)\dif x {\quad\rm for\ any\ }y\in\overline{\Omega}. 
\end{align}
Since the weak solution $f$ satisfies $f\in C^0([0,t];L^2(\D))$ and $\nabla_vf\in L^2(\D_t)$, for fixed small $\delta$, as $k\rightarrow\infty$, we have 
\begin{align}\label{gtf}
\begin{aligned}
f_{\star k}\rightarrow f
{\ \ \rm in\ }C^0([0,t];L^2(\D)),&\\
\nabla_vf_{\star k}\rightarrow \nabla_vf,\ \
(A\nabla_vf)_{\star k}\rightarrow A\nabla_vf,\ \ 
(B\cdot\nabla_vf)_{\star k}\rightarrow B\cdot\nabla_vf
{\ \ \rm in\ }L^2(\D_t),&\\
(cf)_{\star k}\rightarrow cf ,\ \ (G_1)_{\star k}\rightarrow G_1,\ \ (G_0)_{\star k}\rightarrow G_0
{\ \ \rm in\ }L^2(\D_t), &
\end{aligned}
\end{align}
where we remark that these convergences depend only on the exterior cone condition of $\Omega$ (see for instance \cite[Theorem 2.4]{BLD}), and hence they are independent of $\delta$. For any fixed $y\in\Omega$ and $\phi\in C_c^1(\overline{\D_t})$ such that $\phi=0$ on $[0,T]\times\Omega\times\p\V$, we pick the test function $\phi(t,y,v)\rho_k\big(y-2k^{-1} n_y^\delta-x\big)$ for \eqref{existence}. Observing that this function vanishes on $[0,t]\times\p\D$, we derive 
\begin{align}\label{rk0}
\begin{aligned}
\int_{\{t\}\times\D}f_{\star k}\phi-\int_{\{0\}\times\D}f_{\star k}\phi
-\int_{t,y,v}f_{\star k}\p_t\phi
-\int_{t,y,v}vf_{\star k}\cdot\nabla_y\phi + r_k &\\
= \int_{t,y,v}\big[ -(A\nabla_vf+G_1)_{\star k}\cdot\nabla_v\phi
+\phi(B\cdot\nabla_vf)_{\star k} + (cf)_{\star k}\phi + (G_0)_{\star k}\phi\big]&,  
\end{aligned}
\end{align}
where the remainder term $r_k$ is defined by 
\begin{align*}
\begin{aligned}
r_k:= \int_{t,y,v}vf_{\star k}\cdot\nabla_y\phi
-\int_{t,x,y,v} \phi(t,y,v) f(t,x,v)(v\cdot\nabla_x)\rho_k\big(y-2k^{-1}n_y^\delta-x\big).
\end{aligned}
\end{align*}
To acquire the equation satisfied by $f_{\star k}$, we use integration by parts to get 
\begin{align*}
\begin{aligned}
r_k=\int_{[0,t]\times\p\Omega\times\V} (n_y\cdot v) f_{\star k}(t,y,v) \phi(t,y,v)\\
-\int_{t,x,y,v} \phi(t,y,v)f(t,x,v) (v\cdot\nabla_y+v\cdot\nabla_x) \rho_k\big(y-2k^{-1}n_y^\delta-x\big)&.
\end{aligned}
\end{align*}
By the smoothness of $n_y^\delta$, and Young's inequality, we deduce that for fixed $\delta$, as $k\rightarrow\infty$, 
\begin{align*}
\int_{x}  f(t,x,v) (v\cdot\nabla_y + v\cdot\nabla_x) \rho_k\big(y-2k^{-1}n_y^\delta-x\big)&\\
=-k^{-1}\! \int_{x}f(t,x,v) \big(\nabla_y\otimes n_y^\delta\big) v\cdot (\nabla_y\rho_k)\big(y-2k^{-1}n_y^\delta-x\big)&\rightarrow0
{\quad\rm in\ }L_{loc}^2(\overline{\D_T}). 
\end{align*}
Combining this with \eqref{gtf}, we are able to pass to the limit in \eqref{rk0} except for the boundary term. Indeed, we conclude that, there are some functions $R_{1k},R_{0k}\in L^2(\D_T)$ such that $R_{1k},R_{0k}\rightarrow 0$ in $L^2(\D_T)$ as $k\rightarrow0$, and 
\begin{align}\label{FPk}
(\p_t + v\cdot\nabla_x)f_{\star k}
= \nabla_v\cdot(A\nabla_vf_{\star k}+R_{1k}+ G_1) + B\cdot\nabla_vf_{\star k} +cf_{\star k} +G_0 +R_{0k}.  
\end{align}
Consider the test function $\phi=\chi'(f_{\star k})\varphi$, where $\chi\in C^{1,1}(\R)$ such that $\chi'(f)=0$ on $[0,T]\times\Omega\times\p\V$ and $\chi(\iota)=O(\iota^2)$, $\chi''(\iota)=O(1)$ as $|\iota|\rightarrow\infty$, and $\varphi\in C^1_c(\overline{\D_t})$. As $k\rightarrow\infty$, we derive the renormalization formula
\begin{align}\label{renk}
\begin{aligned}
\int_{\{t\}\times\D} \chi(f)\varphi  -\int_{\{0\}\times\D}\chi(f)\varphi  
+ \lim_{k\rightarrow\infty}\int_{[0,t]\times\p\Omega\times\V}  (n_x\cdot v) \chi(f_{\star k})\varphi&\\
=\int_{\D_t} \big[\chi(f)(\p_t+v\cdot\nabla_x)\varphi-A\nabla_v\chi(f)\cdot\nabla_v\varphi 
-\varphi\chi''(f)(A\nabla_vf+G_1)\cdot\nabla_vf&\\
+ \varphi B\cdot\nabla_v\chi(f)  + cf\chi'(f)\varphi -\chi'(f)G_1\cdot\nabla_v\varphi + G_0\chi'(f)\varphi \big]&. 
\end{aligned}
\end{align}
It thus suffices to show the convergence from $\gamma_x f_{\star k}$ to $\gamma_x f$. To this end, we integrate the equation satisfied by $f_{\star j}-f_{\star k}$ (see \eqref{FPk}) against $n_x^\delta\cdot v\,\eta_R(v)(f_{\star j}-f_{\star k})$, where the cut-off function $\eta_R(v)\in C_c^\infty(B_{2R})$ valued in $[0,1]$ satisfies $\eta_R|_{B_R}\equiv1$ with the constant $R>0$. Then, for any fixed $\delta,R>0$, the passage to limit $j,k\rightarrow\infty$ yields that
\begin{align}\label{jk}
\begin{aligned}
\lim_{j,k\rightarrow\infty}\int_{[0,T]\times\p\Omega\times\V}
\big[ (n_x\cdot v)^2 + \big(n_x^\delta-n_x\big)\cdot v\,(n_x\cdot v) \big] 
(f_{\star j}-f_{\star k})^2\eta_R = 0.
\end{aligned}
\end{align}

Let us now additionally assume that $f\in L^\infty(\D_T)$. Then, for any $t\in(0,T]$, we have 
\begin{align*}
\|f_{\star k}\|_{L^\infty(\p\D_t)} \le \|f_{\star k}\|_{L^\infty(\D_T)} \le \|f\|_{L^\infty(\D_T)}. 
\end{align*}
After extracting a subsequence in $k\rightarrow\infty$, there is some function $f_\p\in L^\infty(\p\D_t)$ such that
\begin{align}\label{gxf}
\begin{aligned}
f_{\star k}\stackrel{\ast}{\rightharpoonup}f_\p {\quad\rm in\ } L^\infty(\p\D_t),&\\
\|f_\p\|_{L^\infty(\p\D_t)} 
\le \liminf\nolimits_{k\rightarrow\infty}\|f_{\star k}\|_{L^\infty(\p\D_t)} 
\le \|f\|_{L^\infty(\D_T)}.&
\end{aligned}
\end{align}
Hence, for any fixed $\delta>0$, $(n_x\cdot v)(f_{\star j}-f_{\star k})^2$ is weakly convergent in $L_{loc}^1([0,T]\times\p\Omega\times\V)$ as $j,k\rightarrow\infty$. We also notice that $\big(n_x^\delta-n_x\big) \cdot v$ is locally bounded and converges to $0$ almost everywhere in $[0,T]\times\p\Omega\times\V$ as $\delta\rightarrow 0$. We thus conclude the convergence of $(n_x\cdot v)f_{\star k}$ in $L_{loc}^2([0,T]\times\p\Omega\times\V)$ by sending $j,k\rightarrow\infty$ and $\delta\rightarrow0$ in \eqref{jk}. Together with \eqref{gxf}, choosing the constant $R$ such that $\{v:\varphi\neq0\}\subset B_R$, the formula \eqref{renk} is recast as \eqref{ren}. We point out that $f_\p$ coincides with $\gamma_xf$ on $[0,T]\times\p_x\D$ owing to the uniqueness of the trace. 

Finally, one can remove the assumption that $f\in L^\infty(\D_T)$ by applying the monotony argument presented in \cite[Proof of Theorem~4.5]{Mischler3}. 
\end{proof}

We are also able to get the uniqueness result to weak solutions when $\D=\OO=\Omega\times\R^d$, even if the boundedness of their traces in $L^2(\Sigma_T,|n_x\cdot v|)$ is not known. Here we recall that $\Sigma_t=[0,t]\times\Gamma=[0,t]\times\p\Omega\times\R^d$ and $\Sigma_t^\pm=[0,t]\times\Gamma_\pm$. 

\begin{corollary}[Uniqueness]\label{unique}
Let $(f_1,\gamma f_1)$ and $(f_2,\gamma f_2)$ be two weak solutions to \eqref{existence} in $\OO_T=(0,T)\times\Omega\times\R^d$ such that $f_1=f_2$ on $\{0\}\times\OO$ with $\p\Omega\in C^{0,1}$. Then, $f_1=f_2$ in $\OO_T$, provided that either their traces $\gamma f_1$ and $\gamma f_2$ coincide on $\Sigma_T^-$, or both of $\gamma f_1$ and $\gamma f_2$ satisfy the specular reflection boundary condition. 
\end{corollary}

\begin{proof}
By subtraction, it suffices to consider the weak solution $f:=f_1-f_2$ to \eqref{existence} in $\OO_T$ with $|G_1|=G_0=f|_{\{0\}\times\OO}=0$. Let $R>0$, and $\eta_R(v)\in C_c^\infty(B_{2R})$ be a radial function valued in $[0,1]$ such that $\eta_R|_{B_R}\equiv1$. In view of Lemma~\ref{trace}, we pick $\chi(\iota)=\iota^2$ and $\varphi=\eta_R$ in the formula~\eqref{ren}, where the boundary term reads 
\begin{align}\label{bterm}
\int_{\Sigma_t} (n_x\cdot v) (\gamma f)^2\eta_R. 
\end{align}
If $f|_{\Sigma_T^-}=0$, then \eqref{bterm} is nonnegative. If the specular reflection boundary condition holds for $f_1,f_2$, then so do $f^2\eta_R$; and thus \eqref{bterm} vanishes. Hence, as $R\rightarrow\infty$, we deduce that 
\begin{align*}
\int_{\{t\}\times\OO}f^2 \le 2\int_{\OO_t}\left(-A\nabla_vf\cdot\nabla_vf +fB\cdot\nabla_vf +cf^2\right). 
\end{align*}
By the Cauchy–Schwarz inequality and Gr\"onwall's inequality, we have $f=0$ in $\OO_T$, and hence the uniqueness follows.  
\end{proof}

\subsection{Existence of weak solutions}\label{EW}
\subsubsection{Inflow boundary value problems}\label{subinflow}
The proof of the existence result for inflow problems of \eqref{existence} is inspired from \cite{Boyer} which adopted the idea of vanishing viscosity associated with an appropriate boundary condition for transport equations. Let us stress that our notion of weak solution (see Definition~\ref{def-weak}) is characterized by a pair of functions $(f,\gamma_xf)$, and therefore, it is always necessary to demonstrate the existence of these two functions satisfying the weak formulation~\eqref{weak}. 

\begin{lemma}\label{weakexistence}
Suppose that the boundary $\p\D$ of the domain $\D\subset\R^d_x\times\R^d_v$ is $C^{0,1}$ and consists only of the boundary $\p_x\D$ with respect to $x$ and the boundary $\p_v\D$ with respect to $v$. Let the function $g:\overline{\D_T}\rightarrow\R$ be such that $g,\nabla_vg,(\p_t+v\cdot\nabla_x)g\in L^2(\D_T)$, $g|_{t=0}\in L^2(\D)$ and $g|_{[0,T]\times\p_x\D}\in L^2([0,T]\times\p_x\D,(n_x\cdot v)_-)$. Then, there exists a weak solution $(f,\gamma_x f)$ to \eqref{existence} in $\D_T$ associated with $f=g$ on $\GG$, meaning that $f|_{t=0}=g|_{t=0}$ and $\gamma_xf|_{\{[0,T]\times\p_x\D:\,n_x\cdot v<0\}}=g|_{\{[0,T]\times\p_x\D:\,n_x\cdot v<0\}}$. 
\end{lemma}

\begin{proof}
Without loss of generality, we assume $g=0$; otherwise, we consider the function $f-g$. Let us fix $\varepsilon\in(0,1)$ and consider the weak solution $f_\varepsilon$ to the initial-boundary problem for parabolic equation, 
\begin{align*}
\left\{ 
\begin{aligned}
\ &(\p_t+v\cdot\nabla_x)f_\varepsilon
=\varepsilon\Delta_xf_\varepsilon + \nabla_v\cdot(A\nabla_v f_\varepsilon) + B\cdot\nabla_vf_\varepsilon  +cf_\varepsilon + \nabla_v\cdot G_1 + G_0
{\quad \rm in \ } \D_T,  \\
\ &\;f_\varepsilon= 0 {\quad \rm on \ }  (\{0\}\times\D)\cup([0,T]\times\p_v\D), \\
\ &\; \varepsilon n_x\cdot\nabla_xf_\varepsilon + (n_x\cdot v)_-f_\varepsilon = 0 {\quad \rm on \ } [0,T]\times\p_x\D. \\
\end{aligned}
\right. 
\end{align*}
One may refer to \cite[III. \S5]{Lady} for the classical existence result of the above problem. In the weak formulation, for any $\varphi\in C^1(\overline{\D_t})$ with $t\in(0,T]$ such that  $\varphi=0$ on $[0,t]\times\p_v\D$, 
\begin{align}\label{FPeweak}
\begin{aligned}
\int_{\{t\}\times\D}f_\varepsilon\varphi
+ \int_{[0,t]\times\p_x\D} (n_x\cdot v)_+f_\varepsilon\varphi -\int_{\D_t}f_\varepsilon(\p_t+v\cdot\nabla_x)\varphi &\\
=\int_{\D_t} (-A\nabla_vf_\varepsilon\cdot\nabla_v\varphi-\varepsilon\nabla_xf_\varepsilon\cdot\nabla_x\varphi 
+ \varphi B\cdot\nabla_vf_\varepsilon + c f_\varepsilon\varphi - G_1\cdot\nabla_v\varphi + G_0\varphi)&. 
\end{aligned}
\end{align}
The energy estimate is derived by choosing the solution $f_\varepsilon$ itself for testing, and applying the Cauchy–Schwarz inequality and Gr\"onwall's inequality, which reads
\begin{align*}
\begin{aligned}
\sup_{t\in[0,T]} \int_{\{t\}\times\D} f_\varepsilon^2
+ \int_{[0,T]\times\p_x\D} |n_x\cdot v| f_\varepsilon^2 
+ \int_{\D_T} \left(|\nabla_vf_\varepsilon|^2  + \varepsilon|\nabla_xf_\varepsilon|^2\right)
\lesssim \int_{\D_T} \left(|G_1|^2 +G_0^2\right). 
\end{aligned}
\end{align*}
Similarly, for $\varepsilon_k:=k^{-4}$ and $k\in\N_+$, choosing test function $f_{\varepsilon_{k}}-f_{\varepsilon_{k+1}}$ in the weak formulations satisfied by $f_{\varepsilon_{k}}$ and $f_{\varepsilon_{k+1}}$ yields that 
\begin{align*}
\begin{aligned}
\sup_{t\in[0,T]} \int_{\{t\}\times\D} (f_{\varepsilon_{k}}-f_{\varepsilon_{k+1}})^2 
+ \int_{[0,T]\times\p_x\D} |n_x\cdot v| (f_{\varepsilon_{k}}-f_{\varepsilon_{k+1}})^2 
+ \int_{\D_T}|\nabla_v(f_{\varepsilon_{k}}-f_{\varepsilon_{k+1}})|^2  &\\
+ \int_{\D_T}|\nabla_x(\sqrt{\varepsilon_{k}}f_{\varepsilon_{k}}-\sqrt{\varepsilon_{k+1}}f_{\varepsilon_{k+1}})|^2
\lesssim \left(\sqrt{\varepsilon_{k}}-\sqrt{\varepsilon_{k+1}}\right)^2
\int_{\D_T}|\nabla_xf_{\varepsilon_{k}}||\nabla_xf_{\varepsilon_{k+1}}|&\\
\le \left(\sqrt{\varepsilon_{k}}-\sqrt{\varepsilon_{k+1}}\right)^2\sqrt{\varepsilon_{k}\varepsilon_{k+1}}^{-1}
\int_{\D_T}\left(\varepsilon_{k}|\nabla_xf_{\varepsilon_{k}}|^2+\varepsilon_{k+1}|\nabla_xf_{\varepsilon_{k+1}}|^2\right)&, 
\end{aligned}
\end{align*}
where $(\sqrt{\varepsilon_{k}}-\sqrt{\varepsilon_{k+1}})^2\sqrt{\varepsilon_{k}\varepsilon_{k+1}}^{-1}=k^{-2}(k^2+1)^{-1}$ tends to zero as $k\rightarrow\infty$. Hence, there is some function $f\in C^0([0,T];L^2(\D))$ satisfying $\nabla_vf\in L^2\left(\D_T\right)$ and $f|_{t=0}=f|_{[0,T]\times\p_v\D}=0$, and some function $\gamma_xf\in L^2((0,T);L^2(\p_x\D,|n_x\cdot v|\dif x\dif v))$,  such that as $\varepsilon=\varepsilon_k\rightarrow0$, 
\begin{align}\label{felimit}
\begin{aligned}
f_\varepsilon\rightarrow f {\quad\rm in\ }  C^0([0,T];L^2(\D)), &\\
\nabla_vf_\varepsilon\rightarrow \nabla_vf, \ \; \varepsilon\nabla_xf_\varepsilon\rightarrow 0 {\quad\rm in\ } L^2(\D_T), & \\
f_\varepsilon\rightarrow \gamma_xf{\quad\rm in\ } L^2((0,T);L^2(\p_x\D,|n_x\cdot v|)).  &
\end{aligned}
\end{align}
Sending $\varepsilon\rightarrow0$ in \eqref{FPeweak}, we deduce that the weak formulation~\eqref{weak} holds for the limiting function $f$. 
	
Now we have to show that $\gamma_xf|_{\{n_x\cdot v<0\}}=0$. To this end, we rewrite \eqref{FPeweak} and its limit in $L^2(\D_T)$ as
\begin{align*}
\div(f_\varepsilon,\, vf_\varepsilon-\varepsilon\nabla_xf_\varepsilon,\, -A\nabla_v f_\varepsilon -G_1)
= B\cdot\nabla_vf_\varepsilon +cf_\varepsilon+G_0&\\
\rightarrow B\cdot\nabla_vf +cf+G_0
=\div(f,\, vf,\, -A\nabla_v f -G_1) &. 
\end{align*}
It implies the convergence in $H^{-1/2}(\p\D_T)$ that 
\begin{align*}
n_z\cdot(f_\varepsilon,\,vf_\varepsilon - \varepsilon\nabla_xf_\varepsilon,\,-A\nabla_v f_\varepsilon -G_1 )
\rightarrow n_z\cdot(f,\,vf,\,-A\nabla_v f -G_1 ),  
\end{align*}
where $n_z\in\R^{1+2d}$ is the unit outward normal vector at $z\in\p\D_T$. Combining this with the boundary condition of $f_\varepsilon$ on $[0,T]\times\p_x\D$ and the limiting process of $\gamma_xf$ above, we arrive at $\gamma_xf|_{\{n_x\cdot v<0\}}=0$. This completes the proof. 
\end{proof} 

Furthermore, one is able to show that the weak solution constructed above satisfies the renormalization formula and thus the maximum principle. 

\begin{lemma}\label{max}
Let $(f,\gamma_xf)$ be the weak solution constructed in Lemma~\ref{weakexistence}. 
Then, for any convex $\chi\in C^{1,1}(\R)$ such that $\chi'(f)=0$ on $[0,T]\times\p_v\D$ and $\chi(\iota)=O(\iota^2)$ as $|\iota|\rightarrow\infty$, and any nonnegative $\varphi\in C^1_c(\overline{\D_t})$ with $t\in(0,T]$, we have 
\begin{align}\label{rene}
\begin{aligned}
\int_{\{t\}\times\D}\chi(f)\varphi-\int_{\{0\}\times\D}\chi(f)\varphi  
-\int_{[0,t]\times\p_x\D}(n_x\cdot v)\chi(\gamma_xf)\varphi&\\
\le \int_{\D_t} \big[\chi(f)(\p_t+v\cdot\nabla_x)\varphi 
-A\nabla_v\chi(f) \cdot\nabla_v\varphi - \varphi\chi''(f)( A\nabla_vf+G_1)\cdot\nabla_vf&\\
+ \varphi B\cdot\nabla_v\chi(f) + c f\chi'(f)\varphi -\chi'(f) G_1\cdot\nabla_v\varphi + \chi'(f)G_0\varphi\big]&. 
\end{aligned}
\end{align}
In particular, if additionally $g\in L^\infty(\GG)$, then 
\begin{align}\label{maxl2}
\|f\|_{C^0([0,T];L^2(\D))}\lesssim \|G_1\|_{L^2(\D_T)} +\|G_0\|_{L^2(\D_T)} +\|g\|_{L^\infty(\SMALL\GG)}; 
\end{align} 
moreover, if $G_1=0$ and $G_0\in L^\infty(\D_T)$, then we have the maximum principle that 
\begin{align}\label{maxgg}
\|\gamma_xf\|_{L^\infty([0,T]\times\p_x\D)}
\le \|f\|_{L^\infty(\D_T)} 
\lesssim \|G_0\|_{L^\infty(\D_T)} +\|g\|_{L^\infty(\SMALL\GG)}.  
\end{align}
\end{lemma}

\begin{proof}
Based on the same approximation mechanism of $f$ through $f_\varepsilon$ as in the proof of Lemma~\ref{weakexistence} above, the renormalization formula for the solution $f_\varepsilon$ of the parabolic equation (see its weak formulation \eqref{FPeweak}) is given by choosing the test function $\chi'(f_\varepsilon)\varphi$, with $\chi\in C^{1,1}(\R)$ such that $\chi'(f)=0$ on $[0,T]\times\p_v\D$ and $\chi(\iota)=O(\iota^2)$, $\chi''(\iota)=O(1)$ as $|\iota|\rightarrow\infty$, and $\varphi\in C_c^1(\overline{\D_t})$. More precisely, we have 
\begin{align*}
\begin{aligned}
\int_{\{t\}\times\D}\chi(f_\varepsilon)\varphi-\int_{\{0\}\times\D}\chi(f_\varepsilon)\varphi
+ \int_{[0,t]\times\p_x\D} (n_x\cdot v) \chi(f_\varepsilon)\varphi &\\
=\int_{\D_t} \big[ \chi(f_\varepsilon)(\p_t+v\cdot\nabla_x)\varphi -A\nabla_v\chi(f_\varepsilon)\cdot\nabla_v\varphi -\varepsilon\nabla_x\chi(f_\varepsilon)\cdot\nabla_x\varphi&\\
-\chi''(f_\varepsilon)\varphi (A\nabla_vf_\varepsilon+G_1)\cdot\nabla_vf_\varepsilon
-\varepsilon\chi''(f_\varepsilon)\varphi|\nabla_xf_\varepsilon|^2&\\
+ \varphi B\cdot\nabla_v\chi(f_\varepsilon) + c f_\varepsilon\chi'(f_\varepsilon) \varphi 
-\chi'(f_\varepsilon)G_1\cdot\nabla_v\varphi + \chi'(f_\varepsilon)G_0\varphi\big]&. 
\end{aligned}
\end{align*}
Provided that $\chi$ is convex and $\varphi$ is nonnegative, the passage to limit $\varepsilon\rightarrow0$ with the aid of \eqref{felimit} then implies \eqref{rene}. 

To show the estimates \eqref{maxl2} and \eqref{maxgg}, we may assume $c\le0$; otherwise, we consider the equation solved by $e^{\Lambda t}f$. Let the constant $M:=\|g\|_{L^\infty(\SMALL\GG)}$. Taking $\chi(\iota):=(\iota-M)_+^2$ and $\varphi=1$ in \eqref{rene}, and applying the Cauchy–Schwarz inequality, we get 
\begin{align*}
\begin{aligned}
\int_{\{t\}\times\D}(f-M)_+^2
\le 2\int_{\D_t} \left[-\chi''(f)(A\nabla_vf+G_1)\cdot\nabla_vf
+ \chi'(f)\,  B\cdot\nabla_vf\right]\\
+2\int_{\D_t} \left[c f\chi'(f) + \chi'(f)G_0\right]\lesssim \int_{\D_t}\left[(f-M)_+^2 +|G_1|^2 +G_0^2\right]&, 
\end{aligned}
\end{align*}
where we also used the fact that $cf\chi'(f)\le 2c\;\!(f-M)_+^2$ to produce the second inequality. By Gr\"onwall’s inequality, we obtain 
\begin{align*}
\|(f-M)_+\|_{C^0([0,T];L^2(\D))}^2 
\lesssim \|G_1\|_{L^2(\D_T)}^2 +\|G_0\|_{L^2(\D_T)}^2. 
\end{align*} 
In particular, when $|G_1|=G_0=0$, it turns out that $(f-M)_+=0$ in $\D_T$. These two consequences provide the upper bounds for $f$ in $L^2(\D_T)$ and in $L^\infty(\D_T)$, respectively. Choosing the function $(-f-M)_+$ for testing with a reduction to nonnegative $c$, we then arrive at \eqref{maxl2}, and also get \eqref{maxgg} in the case that $G_0=0$. 

As far as \eqref{maxgg} with general $G_0\in L^\infty(\D_T)$ is concerned, one may consider the equation solved by the function $\pm f-e^{t}\|G_0\|_{L^\infty(\D_T)}$. We finally remark that the estimate about the trace $\gamma_xf$ in \eqref{maxgg} can be achieved by the same approximation argument as in the proof of Lemma~\ref{trace}; see \eqref{gxf}. This concludes the proof. 
\end{proof}

\begin{corollary}\label{globalin}
Let $\p\Omega\in C^{0,1}$, and the function $g\in L^2(\Ge,\dif\mu)$. Then, there exists a unique weak solution $(f,\gamma f)$ to \eqref{existence} in $\OO_T=(0,T)\times\Omega\times\R^d$ associated with $\gamma f=g$ on $\Ge$, meaning that $f|_{t=0}=g|_{t=0}$ and $\gamma f|_{\Sigma_T^-}=g|_{\Sigma_T^-}$; and it satisfies 
\begin{align}\label{FPenergy}
\begin{aligned}
\|f\|_{C^0([0,T];L^2(\OO))} +\|\nabla_vf\|_{L^2(\OO_T)} +\|f\|_{L^2(\p\OO_T,\;\!\dif\mu)}&\\
\lesssim \|G_1\|_{L^2(\OO_T)} + \|G_0\|_{L^2(\OO_T)} +\|g\|_{L^2(\Ge,\;\!\dif\mu)}&. 
\end{aligned}
\end{align} 
Here we abbreviate the trace of $f$ on $\p\OO_T$ by $f$ itself. Besides, if additionally $G_1=0$, $G_0\in L^\infty(\D_T)$, and $g\in L^\infty(\Ge)$, then we have   
\begin{align}\label{maxglobal}
\|\gamma_xf\|_{L^\infty(\Sigma_T)}
\le\|f\|_{L^\infty(\OO_T)}
\lesssim \|G_0\|_{L^\infty(\OO_T)} +\|g\|_{L^\infty(\Ge)}. 
\end{align}
\end{corollary}

\begin{proof}
The existence of weak solutions follows from the same argument as in the proof of Lemma~\ref{weakexistence}, provided that $g$ is regular in the sense that $g\in L^2(\Omega_T;H^1(\R^d))$ and $(\p_t+v\cdot\nabla_x)g\in L^2(\OO_T)$. On account of this, we pick an approximating sequence of compactly supported smooth functions $g_j$ such that $g_j\rightarrow g$ in $L^2(\Ge,\dif\mu)$. Let $f_j$ be a weak solution to \eqref{existence} in $\OO_T$ associated with $f_j=g_j$ on $\Ge$. In view of Lemma~\ref{trace} or Lemma~\ref{max}, by taking $\chi(\iota)=\iota^2$ and $\varphi=1$ in the renormalization formula~\eqref{ren} or \eqref{rene}, and using the Cauchy–Schwarz inequality and Gr\"onwall's inequality, we have  
\begin{align*}
\begin{aligned}
\sup_{t\in[0,T]}\int_{\{t\}\times\OO}f_j^2 +\int_{\Sigma_T} |n_x\cdot v| f_j^2 +\int_{\OO_T}|\nabla_vf_j|^2
\lesssim \int_{\OO_T}\left(G_1^2+G_0^2\right) +\int_{\Ge}g_j^2\dif\mu, 
\end{aligned}
\end{align*} 
which is the estimate~\eqref{FPenergy} for $f_j$. As a consequence of the linear structure of the equation, we also have 
\begin{align*}
\begin{aligned}
\sup_{t\in[0,T]}\int_{\{t\}\times\OO}(f_{j+1}-f_{j})^2 +\int_{\Sigma_t} |n_x\cdot v| (f_{j+1}-f_{j})^2 +\int_{\OO_t}|\nabla_v(f_{j+1}-f_{j})|^2&\\
\lesssim \int_{\Ge}(g_{j+1}-g_{j})^2\dif\mu&. 
\end{aligned}
\end{align*} 
Sending $j\rightarrow\infty$, we acquire a limiting function $f$ of $f_j$ such that $f|_{\p\OO_t}\in L^2(\p\OO_t,\dif\mu)$ and $f|_{\Ge}=g$; furthermore, it is a weak solution to \eqref{existence} in $\OO_T$ and satisfies \eqref{FPenergy}. The uniqueness of weak solutions and the estimate~\eqref{maxglobal} are direct consequences of Corollary~\ref{unique} and Lemma~\ref{max}, respectively. 
\end{proof}

\subsubsection{Specular reflection boundary value problems}
Based on the existence result in \S\ref{subinflow} above, we construct through an iterative method patterned after \cite{Carrillo}, whose argument also works for a certain class of elastic reflection boundary problems but is of limited use for general diffuse reflection boundary problems. 

Let us recall the specular reflection operator $\RR f(t,x,v)=f(t,x,v-2(n_x\cdot v)\;\!n_x)$ for $(t,x,v)\in\Sigma_T$.

\begin{lemma}\label{a01}
Let $\p\Omega\in C^{0,1}$, $G_1,G_0\in L^2(\OO_T)$, and $\fin\in L^2(\OO)$ with $\OO_T=(0,T)\times\OO$ and $\OO=\Omega\times\R^d$. For any constant $a\in[0,1)$, there exists a unique weak solution $(f,\gamma f)$ to \eqref{existence} associated with the initial-boundary condition $f|_{t=0}=\fin$ in $\OO$ and $\gamma_-f=a\RR f$ in $\Sigma_T^-$; furthermore, it satisfies
\begin{align}\label{dra}
\begin{aligned}
\|f\|_{C^0([0,T];L^2(\OO))} +\|\nabla_vf\|_{L^2(\OO_T)}
+(1-a)\|f\|_{L^2(\p\OO_T,\;\!\dif\mu)}&\\
\lesssim \|G_1\|_{L^2(\OO_T)} + \|G_0\|_{L^2(\OO_T)} + \|\fin\|_{L^2(\OO)} &.
\end{aligned}
\end{align} 
\end{lemma}

\begin{proof}
We may assume that $c\le -C_0$ for a fixed constant $C_0>0$; otherwise, we consider the equation solved by $e^{(\Lambda+C_0)t}f$. In view of Corollary~\ref{globalin}, we acquire a sequence of weak solutions $\{f_n\}_{n\in\N}$ to \eqref{existence} through the iterative scheme of inflow boundary value problems associated with 
\begin{align*}
f_n|_{t=0}=\fin{\quad\rm and\quad}
\gamma_-f_{n+1}=a\RR f_{n}{\quad\rm for\ }n\in\N, \quad
 \gamma_-f_0=0. 
\end{align*}
By the definition of the reflection operator, we have 
\begin{align}\label{D0}
\int_{\Sigma_T} (n_x\cdot v)_-f_n^2
= \int_{\Sigma_T} (n_x\cdot v)_-(a\RR f_{n-1})^2
= a^2\int_{\Sigma_T} (n_x\cdot v)_+f_{n-1}^2. 
\end{align}
Taking $\chi(\iota)=\iota^2$ and $\varphi=1$ in the renormalization formula~\eqref{ren} of Lemma~\ref{trace}, and using the Cauchy-Schwarz inequality, we obtain  
\begin{align*}
\begin{aligned}
\int_{\{t\}\times\OO} f_n^2  - \int_{\OO}\fin^2
+ \int_{\Sigma_t} (n_x\cdot v) f_n^2 &\\
=2\int_{\OO_t} \left(-A\nabla_vf_n \cdot\nabla_vf_n
+ f_nB\cdot\nabla_vf_n + c f_n^2 - G_1\cdot\nabla_vf_n + G_0f_n \right)&\\
\le 2\int_{\OO_t} \left[-\frac{1}{2\Lambda}|\nabla_vf_n|^2 
+ \left(\Lambda^3+c+1\right)f_n^2 + |G_1|^2 + G_0^2\right]&. 
\end{aligned}
\end{align*}
It follows by picking $C_0:=1+\Lambda^3$ that 
\begin{align}\label{D1}
\sup_{t\in[0,T]}\int_{\{t\}\times\OO} f_n^2 + \int_{\OO_T} |\nabla_vf_n|^2
\lesssim \int_{\OO_T}\left(|G_1|^2 + G_0^2\right) + \int_{\OO}\fin^2 + \int_{\Sigma_T} (n_x\cdot v)_-f_n^2,
\end{align}
and 
\begin{align}\label{D2}
\int_{\Sigma_T} (n_x\cdot v)_+f_n^2
\le \int_{\OO_T}\left(|G_1|^2 + G_0^2\right) + \int_{\OO}\fin^2 + \int_{\Sigma_T} (n_x\cdot v)_-f_n^2. 
\end{align}
Applying \eqref{D0} and \eqref{D2} iteratively yields that 
\begin{align}\label{Diterat}
\begin{aligned}
 \int_{\Sigma_T} (n_x\cdot v)_+f_n^2 
\le \sum_{i=0}^{n} a^{2i} \int_{\OO_T}\left(|G_1|^2 + G_0^2\right) + \sum_{i=0}^{n} a^{2i} \int_{\OO}\fin^2  &\\
\le \frac{1}{1-a^2} \int_{\OO_T}\left(|G_1|^2 + G_0^2\right) + \frac{1}{1-a^2} \int_{\OO}\fin^2  &. 
\end{aligned}
\end{align}
Similarly, the function $f_{n}-f_{n-1}$ solves \eqref{existence} associated with $|G_1|=G_0=0$,  
\begin{align*}
(f_{n}-f_{n-1})|_{t=0}=0 {\quad\rm and\quad}
\gamma_-(f_{n+1}-f_n)=a\RR(f_{n}-f_{n-1}){\quad\rm for\ }n\in\N_+, 
\end{align*}
we thus obtain  
\begin{align*}
\sup_{t\in[0,T]}\int_{\{t\}\times\OO} (f_{n}-f_{n-1})^2 + \int_{\OO_T} |\nabla_v(f_{n}-f_{n-1})|^2
+\int_{\Sigma_T} |n_x\cdot v|(f_{n}-f_{n-1})^2&\\
\lesssim \int_{\Sigma_T} (n_x\cdot v)_-(f_{n}-f_{n-1})^2
\le a^{2n}\int_{\Sigma_T} (n_x\cdot v)_-(f_{1}-f_{0})^2&. 
\end{align*}
Therefore, sending $n\rightarrow\infty$, we derive a limiting function $f$ of $f_n$, which solves \eqref{existence} associated with $f|_{t=0}=\fin$ and $\gamma_-f=a\RR f$. We point out that \eqref{dra} is a consequence of the estimates \eqref{D1} and \eqref{Diterat} under the limit process $n\rightarrow\infty$. The uniqueness of weak solutions follows from the same argument of the proof in Corollary~\ref{unique}.  
\end{proof}

One disadvantage in the above argument is the lack of information on the trace of solutions as $a\rightarrow1$. With the aid of the trace result from Lemma~\ref{trace}, we achieve the well-posedness for \eqref{existence} under the reflection boundary condition $\gamma_-f=a\RR f$ for the full range $a\in[0,1]$. 

\begin{corollary}\label{globalre}
Let $\p\Omega\in C^{0,1}$, $G_1=0$, $G_0\in L^2\cap L^\infty(\OO_T)$, and $\fin\in L^2\cap L^\infty(\OO)$ with $\OO_T=(0,T)\times\OO$ and $\OO=\Omega\times\R^d$. For any constant $a\in[0,1]$, there exists a unique weak solution $(f,\gamma f)$ to \eqref{existence} in $\OO_T$ associated with $f|_{t=0}=\fin$ in $\OO$ and $\gamma_-f=a\RR f$ in $\Sigma_T^-$; furthermore, it satisfies
\begin{align}\label{drglobal}
\begin{aligned}
\|f\|_{C^0([0,T];L^2(\OO_T))} + \|\nabla_vf\|_{L^2(\OO_T)} 
\lesssim \|G_0\|_{L^2(\OO_T)} + \|\fin\|_{L^2(\OO)} , \\
\|\gamma f\|_{L^\infty(\Sigma_T)}
\le\|f\|_{L^\infty(\OO_T)}
\lesssim \|G_0\|_{L^\infty(\OO_T)} + \|\fin\|_{L^\infty(\OO)} .
\end{aligned}
\end{align} 
\end{corollary}

\begin{proof}
We may assume that $c$ and $G_0$ are nonpositive in $\OO_T$; otherwise, we consider the equation solved by $e^{\Lambda t}f-e^{(1+\Lambda)t}\|G_0\|_{L^\infty(\OO_T)}$. Based on the proof of Lemma~\ref{a01} with the boundedness estimate given by Lemma~\ref{max}, for any $a\in[0,1)$, there exists a unique bounded weak solution $f_a$ to \eqref{existence} in $\OO_T$ associated with $f_a|_{t=0}=\fin$ and $\gamma_-f_a=a\RR f_a$. In the light of Lemma~\ref{trace}, we take $\chi(\iota)=(\iota-M)_+^2$ and $\varphi=1$ in the formula~\eqref{ren}, for the constant $M:=\|\fin\|_{L^\infty(\OO)}$. 
Taking the boundary condition into account, applying Cauchy-Schwarz inequality and the fact that $f_a(f_a-M)_+\ge (f_a-M)_+^2$, we obtain 
\begin{align*}
\begin{aligned}
\int_{\{t\}\times\OO} (f_a-M)_+^2
\le\int_{\{t\}\times\OO} (f_a-M)_+^2 + \int_{\Sigma_t}(n_x\cdot v)(f_a-M)_+^2&\\
\le 2\int_{\OO_t} \left[(f_a-M)_+^2 + cf_a(f_a-M)_+ \right]
\lesssim \int_{\OO_t} (f_a-M)_+^2&. 
\end{aligned}
\end{align*}
By Gr\"onwall's inequality, we acquire the upper bound that $\|(f_a)_+\|_{L^\infty(\OO_T)}\le M$. Similarly, by taking $\chi(\iota)=(-\iota-M)_+^2$ with a reduction to nonnegative $c$ and $G_0$, we get the lower bound that $\|(f_a)_-\|_{L^\infty(\OO_T)}\le M$. Together with Lemma~\ref{a01} and Lemma~\ref{trace}, we arrive at the estimates~\eqref{drglobal} for $f_a$ with $a\in[0,1)$. 

It remains to deal with the case $a=1$. Let us first assume that $\nabla_vA$, $G_0\phi$, $\fin\phi$ are bounded, where $\phi:=\la v\ra^q$ for some (large) constant $q>0$ to be determined. It is straightforward to check that the functions 
\begin{align*}
F&:=f\phi,\\ 
G_1'&:= -2Af\nabla_v\phi,\\ 
G_0'&:= fA:D_v^2\phi + f(\nabla_vA-B)\cdot\nabla_v\phi +G_0\phi
\end{align*} 
verify 
\begin{align*}
(\p_t+v\cdot\nabla_x)F = \nabla_v\cdot(A\nabla_vF) +B\cdot\nabla_v F +cF
+\nabla_v\cdot G_1' + G_0' {\quad\rm in\ } \OO_T. 
\end{align*} 
Noticing that $\nabla_v\phi=qv\la v\ra^{-2}\phi$ and $D_v^2\phi=q(q-1)v\otimes v\la v\ra^{-4}\phi+qI_d\la v\ra^{-2}\phi$, the functions $\nabla_v\cdot G_1'$ and $G_0'$ are recast as 
\begin{align*}
\nabla_v\cdot G_1'&= \frac{4qA:v\otimes v}{\la v\ra^4}F -\frac{2q\,\trace(A) + 2q\nabla_vA\cdot v}{\la v\ra^2}F 
- \frac{2qA v}{\la v\ra^2}\cdot\nabla_vF,\\
G_0'&= \frac{q(q-1)A:v\otimes v}{\la v\ra^4}F + \frac{q\,\trace(A)+q(\nabla_vA-B)\cdot v}{\la v\ra^2} F +G_0\phi. 
\end{align*} 
It thus turns out that the function $F$ solves 
\begin{align}\label{existenceFa}
(\p_t+v\cdot\nabla_x)F = \nabla_v\cdot(A\nabla_vF) +B'\cdot\nabla_v F +c'F +G_0\phi  {\quad\rm in\ } \OO_T, 
\end{align} 
where the new bounded coefficients $B',c'$ are defined by 
\begin{align*}
B'&:=B- \frac{2q A v}{\la v\ra^2},\\ 
c'&:=c+\frac{q(q+3)A:v\otimes v}{\la v\ra^4} -\frac{q\,\trace(A)+q(\nabla_vA+B)\cdot v}{\la v\ra^2}. 
\end{align*} 
In the same manner as before, for fixed $a\in[0,1)$, we derive the solution $F_a$ to \eqref{existenceFa} associated with $F_a|_{t=0}=\fin\phi$ and $\gamma_-F_a=a\RR F_a$, which satisfies \eqref{drglobal} with $G_0$ and $\fin$ replaced by $G_0\phi$ and $\fin\phi$. Equivalently, we obtain the solution $f_a$ to \eqref{existence} associated with $f_a|_{t=0}=\fin$ and $\gamma_-f_a=a\RR f_a$, satisfying \eqref{drglobal} and
\begin{align}\label{fa12}
\|\gamma f_a\la v\ra\|_{L^2(\Sigma_T)} 
\lesssim \|\gamma f_a\phi\|_{L^\infty(\Sigma_T)} 
\lesssim \|G_0\phi\|_{L^\infty(\OO_T)} + \|\fin\phi\|_{L^\infty(\OO)}, 
\end{align} 
where we may choose $q=d+3$ for the first inequality above. Now we argue by approximation. Let $a_i\in[0,1)$ for $i=1,2$, and $f_{a_i}$ be the solution to \eqref{existenceFa} associated with $f_{a_i}|_{t=0}=\fin$ and $\gamma_-f_{a_i}=a_i\RR f_{a_i}$, which satisfies \eqref{fa12}. Picking $\chi(\iota)=\iota^2$ and $\varphi=1$ in the renormalization formula~\eqref{ren} satisfied by $f_{a_i}$, as well as using the Cauchy-Schwarz inequality, yields that 
\begin{align*}
\int_{\{t\}\times\OO} (f_{a_1}-f_{a_2})^2 + \int_{\Sigma_t}(n_x\cdot v) (f_{a_1}-f_{a_2})^2
+\int_{\OO_t}|\nabla_v(f_{a_1}-f_{a_2})|^2
\lesssim \int_{\OO_t}(f_{a_1}-f_{a_2})^2. 
\end{align*}
We also observe that the boundary term above can be written as 
\begin{align*}
\int_{\Sigma_t}(n_x\cdot v) (f_{a_1}-f_{a_2})^2
= \int_{\Sigma_t^+}(n_x\cdot v) \left[(f_{a_1}-f_{a_2})^2-(a_1f_{a_1}-a_2f_{a_2})^2\right]\\
= \int_{\Sigma_t^+}(n_x\cdot v) \left[(1-a_1^2)(f_{a_1}-f_{a_2})^2-2a_1(a_1-a_2)f_{a_1}(f_{a_1}-f_{a_2})-(a_1-a_2)^2f_{a_2}^2\right]&. 
\end{align*}
Owing to the boundedness estimate~\eqref{fa12} for $f_{a_1},f_{a_2}$, the above boundary term tends to zero as $a_1,a_2\rightarrow1$. We then deduce that the limiting function $f$ of $f_a$ as $a\rightarrow1$ satisfies \eqref{drglobal} and solves \eqref{existenceFa} associated with $f|_{t=0}=\fin$ and $\gamma_-f=\RR f$. 

Next, we have to remove the additional boundedness assumptions on $\nabla_vA$, $G_0\phi$, $\fin\phi$. To this end, we approximate $A$, $G_0$, $\fin$ by $A^j$, $G_0^j$, $\fin^j$ in the sense that $A^j\rightarrow A$ pointwisely,  $(G_0^j,\fin^j)\rightarrow (G_0,\fin)$ strongly in $L^2$ and in the weak-* topology of $L^\infty$ as $j\rightarrow\infty$, where $A^j$, $G_0^j$, $\fin^j$ enjoy the same assumptions as $A$, $G_0$, $\fin$, and additionally $\nabla_vA^j$, $G_0^j\phi$, $\fin^j\phi$ are bounded for each $j$. It produces $f^j$ satisfies \eqref{drglobal} and solves \eqref{existence}, with $A$ and $G_0$ replaced by $A^j$ and $G_0^j$, associated with $f^j|_{t=0}=\fin^j$ and $\gamma_-f^j=\RR f^j$. It follows that $g:=f^{j_1}-f^{j_2}$ verifies  
\begin{align*}
(\p_t+v\cdot\nabla_x)g = \nabla_v\cdot\big(A^{j_1}\nabla_vg\big) +B\cdot\nabla_vg +cg
+ \nabla_v\cdot\big((A^{j_1}-A^{j_2})\nabla_vf^{j_2}\big)  +G_0^{j_1}-G_0^{j_2}.  
\end{align*} 
Let the constant $R\ge1$, and $\eta_R(v)\in C_c^\infty(B_{2R})$ be a radial function valued in $[0,1]$ such that $\eta_R|_{B_R}\equiv1$. Applying the renormalization formula~\eqref{ren} satisfied by $g$ with $\chi(\iota)=\iota^2$ and $\varphi=\eta_R$, and using the Cauchy-Schwarz inequality, we have 
\begin{align*}
\int_{\{t\}\times\OO} g^2 \eta_R + \int_{\Sigma_t}(n_x\cdot v) g^2\eta_R
+\int_{\OO_t}|\nabla_vg|^2\eta_R\\
\lesssim \int_{\OO_t} \left(g^2+\big|G_0^{j_1}-G_0^{j_2}\big|^2 
+ \big|A^{j_1}-A^{j_2}\big|^2\big|\nabla_vf^{j_2}\big|^2 \right) 
+ \int_\OO \big(\fin^{j_1}-\fin^{j_2}\big)^2 \eta_R&. 
\end{align*}
Here we point out that $\int_{\Sigma_t}(n_x\cdot v) g^2\eta_R=0$ by the boundary conditions on $f^{j_1},f^{j_2}$. In view of the boundedness estimates~\eqref{drglobal} for $f^{j_1},f^{j_2}$, we are able to send $R\rightarrow\infty$. Then, by extracting subsequences in $j_1,j_2\rightarrow\infty$, we know that the limiting function $f$ of $f^j$ as $j\rightarrow1$ satisfies \eqref{drglobal} and solves \eqref{existenceFa} associated with $f|_{t=0}=\fin$ and $\gamma_-f=\RR f$. Finally, the uniqueness of weak solutions follows from Corollary~\ref{unique}. 
\end{proof}

\section{Regularity for inflow boundary problems}\label{inflow}
We prove in this section the regularity of solutions to \eqref{FP} associated with the inflow boundary condition. Throughout this section we assume that $\Omega$ is a bounded domain in $\R^d$ with $\p\Omega\in C^{1,1}$. 

Let us first introduce the notion of subsolution we will use intensively. 

\begin{definition}
Let $\D$ be a domain in $\R^{2d}$. We say a function $f$ is a \emph{subsolution} to \eqref{FP} in $\D_T=(0,T)\times\D$, if it satisfies $f\in C^0([0,T];L^2(\D))$ and $\nabla_vf\in L^2(\D_T)$, and for any convex nondecreasing $\chi\in C^{0,1}(\R)$ such that $\chi(\iota)=O(\iota^2)$ as $|\iota|\rightarrow\infty$, and any nonnegative $\varphi\in C^1_c(\D_T)$, 
we have 
\begin{align*}
\begin{aligned}
0\le \int_{\D_T} \big[\chi(f)(\p_t+v\cdot\nabla_x)\varphi -A\nabla_v\chi(f) \cdot\nabla_v\varphi
+ \varphi B\cdot\nabla_v\chi(f) + c f\chi'(f)\varphi + s\chi'(f)\varphi \big].  
\end{aligned}
\end{align*}
\end{definition}

\begin{remark}
Suppose the boundary $\p\D\in C^{0,1}$ consists only of finite boundary portions with respect to $x$ and $v$. In view of Lemma~\ref{max}, one can check by approximation that if $f$ is a weak solution to \eqref{FP} in $\D_T$ with nonpositive value on $\p\D_T$, then after zero extension outside of $\D_T$, the function $f_+$ is a subsolution to \eqref{FP} in $\R^{1+2d}$ with the source term $s$ replaced by $s\mathbbm{1}_{f>0}$. 
\end{remark}

\subsection{Preliminary estimates}
This subsection is devoted to some a priori estimates serving as building blocks in the sequel. Let us first prove the basic energy estimate for weak solutions in the presence of spatial boundaries. 
\begin{lemma}[Local energy estimate]\label{energy}
For any weak solution $f$ to \eqref{FP} in $\OO_T$, and any function $\eta\in C_c^1([0,T]\times\overline{\Omega}\times B_2(v_0))$ valued in $[0,1]$ with $v_0\in\R^d$, we have 
\begin{align*}
\int_{\OO_T}|\nabla_vf|^2\eta^2 
\lesssim\left(1+\|\eta\|_{C^1}^2\right)\int_{\supp\eta}\left(\lvv f^2 +s^2\right)
 +\int_{\Ge} f^2\eta^2 \dif\mu. 
\end{align*} 
\end{lemma}
\begin{proof}
In view of the renormalization formula~\eqref{ren} in Lemma~\ref{trace}, picking $\chi(\iota)=\iota^2$ and $\varphi=\eta^2$ yields that 
\begin{align*}
\begin{aligned}
\int_{\{T\}\times\OO}f^2\eta^2-\int_{\{0\}\times\OO}f^2\eta^2
+\int_{\Sigma_T} (n_x\cdot v) f^2\eta^2
+2\int_{\OO_T} \eta^2 A\nabla_vf\cdot\nabla_vf&\\
=2\int_{\OO_T}\left[-2f\eta A\nabla_vf\cdot\nabla_v\eta +f\eta^2B\cdot\nabla_vf +cf^2\eta^2 +sf\eta^2
+f^2\eta(\p_t+v\cdot\nabla_x)\eta\right]&. 
\end{aligned}
\end{align*}
It then turns out that 
\begin{align*}
\begin{aligned}
\int_{\OO_T} |\nabla_vf|^2\eta^2
\lesssim
\int_{\OO_T}\left(|f||\nabla_vf|  \eta|\nabla_v\eta|
+ |f||\nabla_vf|\eta^2 +f^2\eta^2 +|s||f|\eta^2\right)&\\
+\int_{\OO_T}f^2\eta|(\p_t+v\cdot\nabla_x)\eta|
+\int_{\Ge}f^2\eta^2\dif\mu&. 
\end{aligned}
\end{align*}
Applying the Cauchy–Schwarz inequality, we obtain  
\begin{align*}
\begin{aligned}
\int_{\OO_T} |\nabla_vf|^2\eta^2
\lesssim
\int_{\OO_T}\left[f^2|\nabla_v\eta|^2
+ f^2\eta^2  + s^2 \eta^2
+f^2\eta|(\p_t+v\cdot\nabla_x)\eta|\right]
+\int_{\Ge}f^2\eta^2\dif\mu&, 
\end{aligned}
\end{align*}
which implies the desired result. 
\end{proof}

We then state three lemmas known in the literature. One of the main results in \cite{GIMV} is the following interior regularity estimate for solutions to \eqref{FP}; see \cite[Theorem~1.4]{GIMV}. We have at our disposal its scaled version as follows. Here we recall some notations presented in \S\ref{invariant} and \S\ref{notations}. For $z_0=(t_0,x_0,v_0)$, $Q_r(z_{0})=\big\{ (t,x,v): \, t_{0} - r^{2} < t \le t_{0}, \, |x - x_{0} - (t - t_{0})v_{0}| < r^{3},\, |v - v_{0}| < r\big\}$. A constant is said to be universal if it depends only on $d,T,\Lambda,m,\beta,\Omega,p,q,m,l,\omega,\epsilon$ appearing below. 

\begin{lemma}[Interior H\"older estimate]\label{interior}
There exists a universal constant $\alpha\in(0,1)$ such that for any constants $0<r<R\le1$, and any weak solution $f$ to \eqref{FP} in $Q_R(z_0)$ with $z_0\in\R^{1+2d}$, we have 
\begin{align*}
(R-r)^\alpha[f]_{C^\alpha(Q_r(z_0))}
\lesssim \|f\|_{L^\infty(Q_R(z_0))} +\|s\|_{L^\infty(Q_R(z_0))}. 
\end{align*} 
\end{lemma}

The proof of such H\"older estimate essentially relies on the following two lemmas for subsolutions, that is, the local boundedness estimate \cite[Theorem 3.1]{GIMV} and the oscillation reduction \cite[Lemma~4.5]{GIMV}. The local boundedness estimate from $L^p$ to $L^\infty$ can be rephrased as follows. 
\begin{lemma}[Local boundedness]\label{bdd}
Let the constant $p>0$. For any constant $0<r<R\le1$, and any subsolution $f$ to \eqref{FP} in $Q_R(z_0)$ with $z_0\in\R^{1+2d}$, we have  
\begin{align*}
\|f_+\|_{L^\infty(Q_r(z_0))}
\lesssim (R-r)^{-(2+4d)/p}\|f_+\|_{L^p(Q_R(z_0))} 
+ \|s\|_{L^\infty(Q_R(z_0))}. 
\end{align*} 
\end{lemma}

The oscillation reduction states that if a subsolution is far away from its upper bound in a subset occupying some non-negligible space with a certain time lag, then it cannot get close to this bound in a localized region. 
\begin{lemma}[Oscillation reduction]\label{measure}
Let the constant $\omega\in(0,1)$, and the coefficient $c=0$. Then, there exist some (small) universal constants $\lambda,\varrho,\theta\in(0,1)$ such that for any subsolution $f$ to \eqref{FP} with $f\le1$ and $|s|\le\lambda$ in $Q_1$ satisfying 
\begin{align*}
\big|\{f\le 0\}\cap Q_{2\varrho}^-\big|\ge \omega| Q_{2\varrho}^-|, 
\end{align*}
for the shifted cylinder $Q_{2\varrho}^-:=Q_{2\varrho}\left(-1/2,0,0\right)\subset Q_1$, we have 
\begin{align*}
f\le 1-\theta {\quad in\ }Q_\varrho. 
\end{align*} 
\end{lemma}

\subsection{Local estimates}\label{inflowproof}
The following proposition lies at the core of our results. 

\begin{proposition}\label{local}
Assume that the constants $p\ge2$ and $\beta,\epsilon\in(0,1]$, and the functions $s\in L^\infty(\OO_T)$ and $g\in L^2(\Ge,\dif\mu)\cap L^\infty(\Ge)$. Let $f$ be a weak solution to \eqref{FP} in $\OO_T$ such that $f=g$ on $\Ge$.  
Then, for any $z_0=(t_0,x_0,v_0)\in\overline{\OO_T}$, we have 
\begin{align}\label{bddf}
\begin{aligned}
\|f\|_{L^\infty(\OO_T\cap B_1(z_0))} 
\lesssim \lvv^{\max\{1/2,\;\!(2+4d)/p\}}\|f\|_{L^p(\OO_T\cap B_2(z_0))} +\|s\|_{L^\infty(\OO_T\cap B_2(z_0))} &\\
+\|g\|_{L^2(\Ge\cap B_2(z_0),\;\!\dif\mu)} +\|g\|_{L^\infty(\Ge\cap B_2(z_0))}&;  
\end{aligned} 
\end{align}
if additionally $g\in C^\beta(\Ge)$, then there is some universal constant $\alpha\in(0,1)$ such that 
\begin{align}\label{holderf}
\begin{aligned}
\ [f]_{C^\alpha(\OO_T\cap B_{1}(z_0))}
\lesssim \lvv^{1/2+\epsilon}\|f\|_{L^\infty(\OO_T\cap B_2(z_0))} &\\
+\lvv^{\epsilon}\|s\|_{L^\infty(\OO_T\cap B_2(z_0))} 
+ [g]_{C^\beta(\Ge\cap B_{2}(z_0))}&. 
\end{aligned}
\end{align} 
\end{proposition}

\begin{remark}\label{localremark}
Based on the similar derivation of the estimate~\eqref{bddf}, we also have 
\begin{align*}
\begin{aligned}
\|f\|_{L^\infty(\OO_T\cap B_1(z_0))} 
\lesssim \lvv^{2d}\|f\|_{L_t^\infty L_{x,v}^2\left(\OO_T\cap B_2(z_0)\right)} +\|s\|_{L^\infty(\OO_T\cap B_2(z_0))} &\\
+\|g\|_{L^2(\Ge\cap B_2(z_0),\;\!\dif\mu)} +\|g\|_{L^\infty(\Ge\cap B_2(z_0))}&.  
\end{aligned} 
\end{align*}
\end{remark}

Before starting the proof, let us first set up an appropriate coordinate system, inspired by the one used in \cite{GHJO}. 

\begin{lemma}\label{flattening}
Let $\p\Omega\in C^{1,1}$ and $x_0\in\p\Omega$. There exists some constant $R\in(0,1]$ depending only on $d$ and $\p\Omega$, and some neighborhood $U$ of $x_0$, and some $C^{1,1}$-function $\psi:(-2R,2R)^{d-1} \rightarrow \R$ such that the map $P:(-R,R)^d\rightarrow U$, defined by 
\begin{align}\label{flatten}
\begin{aligned}
P(\check{y},y_d):=\mm(\check{y})+y_d\nn(\check{y},y_d),
\end{aligned} 
\end{align}
is a diffeomorphism from $(-R,R)^d$ to $U$, and from $(-R,R)^{d-1}\times(-R,0)$ to $U\cap\Omega$, where 
\begin{align*}
\check{y}:=(y_1,\ldots,y_{d-1})\in(-R,R)^{d-1},\ \ y_d\in(-R,R), 
\end{align*}
and the maps $\mm:(-R,R)^{d-1}\rightarrow\R^d$ and $\nn:(-R,R)^d\rightarrow \R^d$ are given by 
\begin{align*}
\mm(\check{y})&:=(\check{y},\psi(\check{y}))^T,\\
\nn(\check{y},y_d)&:=|(\psi'(\check{y},y_d),1)|^{-1}(-\psi'(\check{y},y_d),1)^T.
\end{align*} 
Here $\psi':(-R,R)^d\rightarrow\R^d$, which serves as an approximation of $D\psi$, is defined by
\begin{align*}
\psi'(\check{y},y_d):=\left\{ 
\begin{aligned}
\ &|y_d|^{1-d}\int_{\R^{d-1}}\rho(y'/y_d)\,D\psi(\check{y}-y') \dif y' {\quad\rm for\ }y_d\not=0,  \\
\ &D\psi(\check{y}){\quad\rm for\ }y_d=0, \\
\end{aligned}
\right. 
\end{align*} 
where $\rho\in C_c^\infty(\R^{d-1})$ is a fixed nonnegative function such that $\supp\rho\subset B_1\subset\R^{d-1}$ and $\int_{\R^{d-1}}\rho(y')\dif y'=1$. Furthermore, $\nn(\check{y},0)=n_x$ for $x=\mm(\check{y})$ and $P\in C^{1,1}$ in $(-R,R)^d$. 
\end{lemma}

\begin{proof}
In a local coordinate system, we can characterize the boundary portion of $\p\Omega$ near $x_0$ by means of the epigraph of a function $\psi\in C^{1,1}$ defined on $(-2R,2R)^{d-1}$ for some (small) constant $R\in(0,1]$. We now have to check that the map $P$ is well-defined. Denoting by $P'$ the Jacobian matrix of $P$, we know that 
\begin{align}\label{DP}
P'=(D\mm+y_dD_{\check{y}}\nn;\nn+y_dD_{y_d}\nn).
\end{align} 
Since $\psi\in C^{1,1}$, provided that $|y_d|$ is small, the determinant of $P'$ is 
\begin{align}\label{det}
\begin{aligned}
\det(P')
&=|(D\psi,1)|^{-1}\;\! \det\!
\begin{pmatrix}
I_{d-1}  & -(D\psi)^T \\
D\psi & 1
\end{pmatrix}
+O(y_d)\\
&= |(D\psi,1)| + O(y_d). 
\end{aligned}
\end{align}  
Hence, by taking $R$ small enough (depending only on $d$ and $\|\psi\|_{C^{1,1}}$), we have 
\begin{align}\label{detk}
\kappa^{-1}\le \det(P') \le\kappa {\quad\rm in\ } (-R,R)^d,
\end{align} 
for some constant $\kappa>1$ depending only on $d$ and $\|D\psi\|_{L^\infty}$. It follows from the inverse function theorem that the diffeomorphism $P:(-R,R)^d\rightarrow U$, with the neighborhood $U$ of $x_0$, exists as asserted. 

To derive the $C^{1,1}$ estimate of $P$, it suffices to establish the estimate for $y_d\nn$. Indeed, a straightforward calculation using integration by parts shows that the vector $\psi'$, defined through the convolution integral, satisfies 
\begin{align*}
\big\|\big(D_{\check{y}}^2,D_{\check{y}}D_{y_d},D_{y_d}^2\big)(y_d\psi')\big\|_{L^\infty((-R,R)^d)}
\lesssim \|D\psi\|_{C^{0,1}((-2R,2R)^d)}. 
\end{align*} 
This implies the desired $C^{1,1}$ regularity for $P$. 
\end{proof}

We first remark that based on the interior H\"older estimate in Lemma~\ref{interior} and the propagation of H\"older estimate forward in time \cite[Corollary 4.6]{YZ}, it actually suffices to derive the estimate near the phase boundary. Armed with the way of boundary flattening presented in the above lemma, we are able to reduce general boundary problems to a one-dimensional space framework (Step~\hyperref[step1]{1}). After setting up the transformed boundary value problem and using some approximation argument if necessary, we will extend the transformed equation across the singular set $\Gamma_0$ and the portion $\Gamma_+$ where boundary conditions are lost; and this extension is shown to be continuous (Step~\hyperref[step2]{2}). The new problem, with fully prescribed boundary conditions, can be addressed though the analysis of properties for subsolutions with the aid of Lemmas~\ref{bdd}, \ref{measure} (Step~\hyperref[step3]{3}). We then proceed with delicate oscillation estimates, as certain coefficients of the transformed equation tend to be unbounded when the velocity variable goes to infinity (Step~\hyperref[step4]{4}). Finally, we remove the approximation assumption in the concluding step (Step~\hyperref[step5]{5}). 

Let us now turn to the proof in detail. 

\begin{proof}[Proof of Proposition~\ref{local}]
The proof will proceed in five steps. 

\subsubsection*{\textnormal{\textit{Step 1. Localization and boundary flattening}}}\label{step1}\ \\
Let $\Q,\Q_0$ be two open neighborhoods of the point $(x_0,v_0)\in\Gamma$ such that $\overline{\Q}\subset U\times B_1(v_0)$ and $\overline{U}\times\overline{B}_1(v_0)\subset\Q_0$, for $U$ given by Lemma~\ref{flattening}. Take two fixed cut-off functions $\phi\in C^\infty_c(U\times B_1(v_0))$ and $\eta\in C^\infty_c(\Q_0)$ both valued in $[0,1]$ such that $\phi|_{\Q}\equiv1$ and $\eta|_{U\times B_1(v_0)}\equiv 1$. A direct computation yields that the function $F:=f\phi$ satisfies 
\begin{align}\label{FPw}
(\p_t+v\cdot\nabla_x)F = \nabla_v\cdot\left(A\nabla_vF\right) +B\cdot\nabla_v F +cF
+\nabla_v\cdot G_1 + G_0 {\quad\rm in\ } \OO_T, 
\end{align} 
where $G_1,G_0$ are given by 
\begin{align*}
G_1&:=-Af\nabla_v\phi,\\
G_0&:=-(A\nabla_vf+Bf)\cdot\nabla_v\phi +fv\cdot\nabla_x\phi +s\phi. 
\end{align*} 
In particular, $G_1,G_0\in L^2(\OO_T)$ are compactly supported in
\begin{align*}
\U:=[0,T]\times(U\cap\overline{\Omega})\times B_1(v_0),
\end{align*} 
and the localized equation~\eqref{FPw} coincides with the original one~\eqref{FP} in $\U\cap((0,T)\times\Q)$. Applying the local energy estimate given by Lemma~\ref{energy} with $\eta$ picked above, we have 
\begin{align}\label{gg}
\|G_1\|_{L^2(\U)} +\|G_0\|_{L^2(\U)}
\lesssim \lvv^{1/2}\|f\|_{L^2(\U)} +\|s\|_{L^2(\U)}
 +\|g\|_{L^2(\U\cap\Ge,\;\!\dif\mu)}.  
\end{align} 

Let us abbreviate $z=(t,x,v)$ and $\overline{z}=(t,y,w)$. Consider the function $\overline{F}$ with respect to $\overline{z}$ and the transformation $\SS:\U\rightarrow\W:=\SS(\U)$ defined by the prescriptions: 
\begin{align}\label{transformation}
\begin{aligned}
&\overline{F}:=\det\!\left(\frac{\p z}{\p\overline{z}}\right) F\circ\SS^{-1},\\ 
&\SS^{-1}:\,\overline{z}=(t,y,w)\longmapsto  z=(t,x,v):=\left(t,P(y),P'\;\;\!\!\!\!(y)w\right).  
\end{aligned}
\end{align}
The Jacobian matrix $\frac{\p(x,v)}{\p(y,w)}=\left(\begin{smallmatrix}P' & 0\\ D_yv & P'\end{smallmatrix}\right)$, thus $\det\!\left(\frac{\p z}{\p\overline{z}}\right) = (\det(P'))^2$ depends only on the variable $y$ and is nondegenerate in $(-R,R)^d$ due to Lemma~\ref{flattening}. Indeed, it follows from \eqref{detk} that for some universal constant $\kappa>1$, 
\begin{align*}
\kappa^{-2}\le \det\!\left(\frac{\p z}{\p\overline{z}}\right) \le\kappa^2 {\quad\rm for\ any\ } y\in(-R,R)^d. 
\end{align*} 
This shows that $\overline{F}$ is well-defined in $\W$. Moreover, $F$ is supported in $\U$ so that $\overline{F}$ is supported in $\W$. 

Regarding to the boundary condition, it now suffices to consider the data on $\U\cap\Sigma_T^-$. Let $y_d:=y\cdot e_d$ and $w_d:=w\cdot e_d$. Notice that for any $x\in\p\Omega$, the outward normal vector $n_x=\nn(\check{y},0)$. Using \eqref{DP} and the identity $(D\mm)^T\nn=0$, we have 
\begin{align*}
n_x\cdot v=\nn^T(D\mm;\nn)\cdot w=e_d\cdot w=w_d {\quad\rm on\ } \{y_d=0\}, 
\end{align*} 
which also means that 
\begin{align*}
\left\{\pm\, n_x\cdot v<0, \ z\in\U\cap\Sigma_T\right\}
\quad\Longleftrightarrow\quad
\left\{\pm\, w_d<0, \ y_d=0, \ \overline{z}\in\W\right\}. 
\end{align*} 
Therefore, the prescribed boundary value $g$ for $F$ on $\U\cap\Sigma_T^-$ implies the prescribed value $\overline{g}:=\det\!\left(\frac{\p z}{\p\overline{z}}\right)(g\phi)\circ\SS^{-1}$ for $\overline{F}$ on the boundary portion $\{\overline{z}\in\W:y_d=0,w_d<0\}$. 

To derive the equation of $\overline{F}$, we take $\overline{\varphi}\in C_c^1(\W)$ and $\varphi:=\overline{\varphi}\circ\SS$. 
By a change of variables, we have 
\begin{align*}
v\cdot\nabla_x\varphi
&=P'w\cdot\big(P'^{-T}\nabla_y\big)\,\overline{\varphi} +P'w\cdot\left((D_xw)^T\nabla_w\right)\overline{\varphi}\\
&=w\cdot\nabla_y\overline{\varphi} + (D_xw)v\cdot \nabla_w\overline{\varphi}. 
\end{align*} 
It then follows that 
\begin{align*}
\begin{aligned}
\int_{\U\cap(\{t\}\times\OO)} F\varphi -\int_{\U\cap(\{0\}\times\OO)}F\varphi
+\int_{\U\cap\Sigma_T} (n_x\cdot v) F\varphi -\int_{\U} F(\p_t+v\cdot\nabla_x)\varphi  &\\
=\int_{\W\cap\SS(\{t\}\times\OO)} \overline{F}\overline{\varphi}  -\int_{\W\cap\SS(\{0\}\times\OO)}\overline{F}\overline{\varphi} 
+\int_{\W\cap\{y_d=0\}} w_d\overline{F}\overline{\varphi} 
-\int_{\W} \overline{F}(\p_t+w\cdot\nabla_y)\overline{\varphi}  &\\ 
+\int_{\W}\Big[\overline{\varphi}((D_xw)v)\circ\SS^{-1}\! \cdot\nabla_w\overline{F}
+\nabla_w\cdot\big[((D_xw)v)\circ\SS^{-1}\big]\overline{F}\overline{\varphi}\Big]&, 
\end{aligned}
\end{align*}
where we notice by its definition and Lemma~\ref{flattening} that 
\begin{align}\label{BB}
\begin{aligned}
\big\|((D_xw)v)\circ\SS^{-1}\big\|_{L^\infty(\W)}
+\big\|\nabla_w\cdot\big[((D_xw)v)\circ\SS^{-1}\big]\big\|_{L^\infty(\W)}
\lesssim \lvv^2. 
\end{aligned}
\end{align}
In addition, 
\begin{align*}
\int_\U\big(-A\nabla_vF\cdot\nabla_v\varphi +\varphi B\cdot\nabla_vF + cF\varphi
-G_1\cdot\nabla_v\varphi +G_0\varphi\big)&\\
=\int_\W \big[-\big(P'^{-1}A P'^{-T}\big)\circ\SS^{-1} \nabla_w\overline{F}\cdot\nabla_w\overline{\varphi}
+\overline{\varphi}\big(P'^{-1}B\big)\circ\SS^{-1}\!  \cdot\nabla_w \overline{F}&\\
+c\circ \SS^{-1}\overline{F}\overline{\varphi}
-\big(P'^{-1}G_1\big)\circ\SS^{-1}\! \cdot\nabla_w\overline{\varphi} +G_0\circ\SS^{-1}\overline{\varphi} \big]&. 
\end{align*} 
In brief, \eqref{FPw} is equivalent to the following equation 
\begin{align}\label{extendequation}
(\p_t+w\cdot\nabla_y)\overline{F}
=\nabla_w\cdot\big(\overline{A}\nabla_w\overline{F}\big) +\overline{B}\cdot\nabla_w\overline{F} +\overline{c}\overline{F}
+\nabla_w\cdot\overline{G}_1 +\overline{G}_0
{\quad\rm in\ }\W, 
\end{align} 
where the new coefficients are defined in $\W$ by 
\begin{align*}
&\overline{A}:=P'^{-1}(A\circ\SS^{-1})P'^{-T},\\
&\overline{B}:=P'^{-1}(B\circ\SS^{-1})-((D_xw)v)\circ\SS^{-1},\\
&\overline{c}:=c\circ\SS^{-1}\!-\nabla_w\cdot\left[((D_xw)v)\circ\SS^{-1}\right],\\
&\overline{G}_1:=\det\!\left(\frac{\p z}{\p\overline{z}}\right)P'^{-1}G_1\circ\SS^{-1},\\
&\overline{G}_0:=\det\!\left(\frac{\p z}{\p\overline{z}}\right)G_0\circ\SS^{-1}. 
\end{align*} 

\subsubsection*{\textnormal{\textit{Step 2. Extension procedure}}}\label{step2}\ \\
Consider the extended domain 
\begin{align*}
\W^\natural:=\W\cup\left\{\overline{z}\in\SS([0,T]\times U\times B_1(v_0)): w_d\ge0\right\}, 
\end{align*} 
and its effective boundary portion 
\begin{align*}
\p_{\rm eff}\W^\natural
:=\p\W^\natural\cap
\left(\{y_d=0,\,w_d\le0\}\cup\{y_d>0,\,w_d=0\}\cup\{t=0\}\right). 
\end{align*} 
We denote by $\p_y\W^\natural$ and $\p_w\W^\natural$ the boundary portions of $\p\W^\natural$ with respect to $y$ and $w$, respectively. Extend the coefficients $\overline{A},\overline{B},\overline{c},\overline{G}_1,\overline{G}_0$ as $A^\natural,B^\natural,c^\natural,G_1^\natural,G_0^\natural$, respectively, by setting them in $\W^\natural\backslash\W$ as 
\begin{align*}
&A^\natural:=P'^{-1}P'^{-T},\\
&B^\natural:=-((D_xw)v)\circ\SS^{-1},\\
&c^\natural:=-\nabla_w\cdot\left[((D_xw)v)\circ\SS^{-1}\right],\\
&\big|G_1^\natural\big|=G^\natural_0:=0. 
\end{align*} 
Taking note of \eqref{gg} and \eqref{BB}, we see that $G^\natural_1,G^\natural_0\in L^2(\W^\natural)$, and there is some universal constant $K>1$ such that all the eigenvalues of $A^\natural$ lie in $[K^{-1},K]$, and $|B^\natural|, |c^\natural|$ are bounded by $K\lvv^2$. 

As shown in Figure~\ref{img}, this step is devoted to the extension of the solution $\overline{F}$ from the dark region to the checkerboard area, whose boundary value is prescribed as $\overline{g}$ on the black border lines (a section of the effective boundary $\p_{\rm eff}\W^\natural$) and is identically zero near the gray border lines (due to the localization). This procedure relies on the existence result presented in Lemma~\ref{weakexistence} which will not be applied directly, since the result is valid only for some certain regular boundary data. On account of this, we assume $g\in C^1(\overline{\OO_T})$, which will be removed in the final step.  

\begin{figure}
\includegraphics[width=7.8cm]{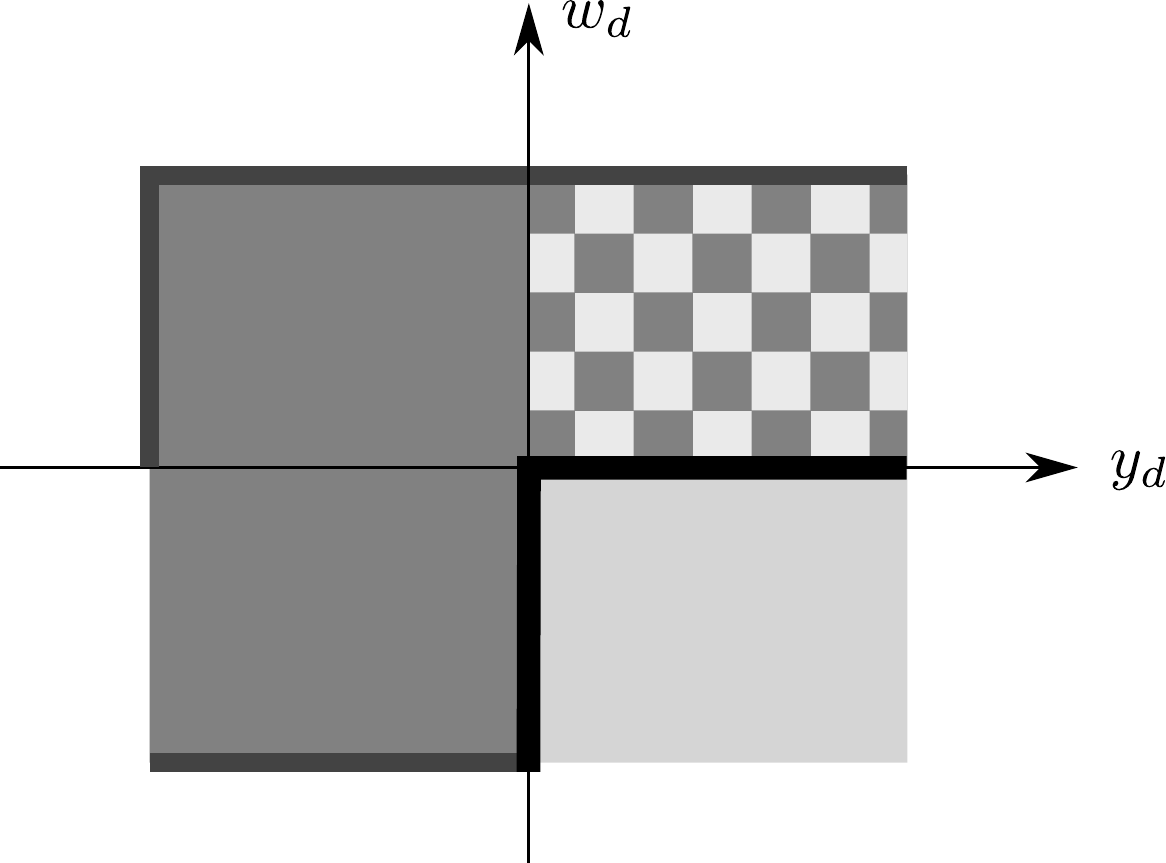}
\caption{The solid dark region is a section of $\W$. 
The corresponding section of $\W^\natural$ consists of regions with dark color and checkerboard pattern.}\label{img}
\end{figure}

Based on the assumption that $g\in C^1(\overline{\OO_T})$, by taking $\check{y}:=(y_1,\ldots,y_{d-1})$ and extending 
\begin{align*}
g^\natural(\overline{z}):=\overline{g}(\overline{z})
=\det\!\left(\frac{\p z}{\p\overline{z}}\right)(g\phi)\circ\SS^{-1}(\overline{z})  {\quad\rm in\ }\{y_d\le 0\},\\
g^\natural(\overline{z}):=\det\!\left(\frac{\p z}{\p\overline{z}}\right)
g\circ\SS^{-1}(t,\check{y},0,w)\,\phi\circ\SS^{-1}(\overline{z})  {\quad\rm in\ } \{y_d>0\}, 
\end{align*} 
the function $\overline{g}$ is extended to a Lipschitz function $g^\natural\in\R^{1+2d}$. 
In view of \eqref{extendequation} and Lemma~\ref{weakexistence}, we get a weak solution $F^\natural\in L^2(\W^\natural)$ by means of solving the problem  
\begin{align*}
(\p_t+w\cdot\nabla_y)F^\natural
=\nabla_w\cdot\big(A^\natural\nabla_wF^\natural\big) +B^\natural\cdot\nabla_wF^\natural +c^\natural F^\natural
+\nabla_w\cdot G^\natural_1 +G^\natural_0
{\quad\rm in\ }\W^\natural
\end{align*}
associated with the boundary condition 
\begin{align}\label{wboundary}
\begin{aligned}
F^\natural=g^\natural {\quad\rm on\ } \p_{\rm eff}\W^\natural,\\
F^\natural=0 {\quad\rm on\ }  \big(\p_y\W^\natural\backslash\{y_d=0\}\big) \cup \big(\p_w\W^\natural\backslash\{w_d=0\}\big). 
\end{aligned}
\end{align} 
Taking the localization $\phi|_\Q\equiv1$ into account, we achieve in $\W^\natural\cap\SS((0,T)\times\Q)$ the equation 
\begin{align}\label{localequation}
\begin{aligned}
(\p_t+w\cdot\nabla_y)F^\natural
=\nabla_w\cdot\big(A^\natural\nabla_wF^\natural\big) +B^\natural\cdot\nabla_wF^\natural
+ c^\natural F^\natural +s^\natural
\end{aligned}
\end{align}
associated with \eqref{wboundary}, where we set the new source term $s^\natural:=\det\!\left(\frac{\p z}{\p\overline{z}}\right)s\circ\SS^{-1}$ in $\W$ which is extended identically zero outside of $\W$. 

\subsubsection*{\textnormal{\textit{Step 3. Local boundedness estimate}}}\label{step3}\ \\
We now observe that the function $F^\natural\circ\SS$ in turn solves the following equation in $\SS^{-1}(\W^\natural)$, 
\begin{align}\label{inturn}
\begin{aligned}
(\p_t+v\cdot\nabla_x)(F^\natural\circ\SS)&\\
=\nabla_v\cdot\big(\underline{A}\nabla_v(F^\natural\circ\SS)\big) + \underline{B}\cdot\nabla_v(F^\natural\circ\SS)
+\underline{c}F^\natural\circ\SS +\nabla_v\cdot\underline{G}_1 +\underline{G}_0&,
\end{aligned}
\end{align}
where $\underline{A}:=A$, $\underline{B}:=B$, $\underline{c}:=c$, $\underline{G}_1:=G_1$, $\underline{G}_0:=G_0$ in $\U$; meanwhile, $\underline{A}:=I_d$, and $\underline{B}$, $\underline{c}$, $\underline{G}_1$, $\underline{G}_0$ are identically zero in $\SS^{-1}(\W^\natural\backslash\W)$. 

By virtue of Lemma~\ref{max} and the estimate~\eqref{gg}, we have
\begin{align}\label{l2}
\begin{aligned}
\|F^\natural\circ\SS\|_{L^2(\SS^{-1}(\W^\natural))}
\lesssim \|\underline{G}_1\|_{L^2(\W)} +\|\underline{G}_0\|_{L^2(\W)} +\|g^\natural\|_{L^\infty(\W\cap\SS(\Ge))} &\\
\lesssim \lvv^{1/2}\|f\|_{L^2(\U)} +\|s\|_{L^2(\U)}
+\|g\|_{L^2(\U\cap\Ge,\;\!\dif\mu)} +\|g\|_{L^\infty(\U\cap\Ge)}&. 
\end{aligned}
\end{align} 
By setting 
\begin{align*}
M&:=\|g^\natural\circ\SS\|_{L^\infty(\U\cap\Ge)},\\
\underline{F}&:=\big(F^\natural\circ\SS-M\big)_+, 
\end{align*}
the function $\underline{F}$ vanishes on the boundary portion $\SS^{-1}(\p_{\rm eff}\W^\natural)$. As a consequence of the zero extension for $\underline{F}$ to the region $((0,T)\times U\times B_1(v_0))\backslash\SS^{-1}(\W^\natural)$, and the localization property from $\phi$, the function $\underline{F}$ becomes a subsolution verifying 
\begin{align*}
(\p_t+v\cdot\nabla_x)\underline{F}
\le \nabla_v\cdot\big(\underline{A}\nabla_v\underline{F}\big) + \underline{B}\cdot\nabla_v\underline{F} +|\underline{c}|\underline{F} +|\underline{c}|M +\underline{s}
{\quad\rm in\ } (-1,T)\times\Q. 
\end{align*}
In this manner, every boundary point reduces to the interior one. Let us take $z_0\in\Sigma_T\cup(\{0\}\times\OO)$, and pick the constant $R_0\in(0,1]$ such that $R_0\approx\lvv^{-1}$ and $Q_{2R_0}(z_0)\subset(-1,T]\times\Q$. Applying Lemma~\ref{bdd}, along with \eqref{l2}, we derive the boundedness of $F^\natural\circ\SS$ from above that 
\begin{align*}
\begin{aligned}
\sup\nolimits_{Q_{R_0}(z_0)}F^\natural\circ\SS
&\le \sup\nolimits_{Q_{R_0}(z_0)}\underline{F} + M\\
&\lesssim \|F^\natural\circ\SS\|_{L^2(\SS^{-1}(\W^\natural))} +\|s\|_{L^\infty(\U)}+  M \\
&\lesssim \lvv^{1/2}\|f\|_{L^2(\U)} +\|s\|_{L^\infty(\U)} 
+\|g\|_{L^2(\U\cap\Ge,\;\!\dif\mu)} +M.  
\end{aligned}  
\end{align*}  
Similarly as regards the zero extension for the function $\big(\!-F^\natural\circ\SS-M\big)_+$, we have the boundedness of $F^\natural\circ\SS$ from below. Hence, for any $z_0\in\Sigma_T\cup(\{0\}\times\OO)$, 
\begin{align}\label{bddF}
\begin{aligned}
\|f\|_{L^\infty(\U\cap Q_{R_0}(z_0))} 
&\le \|F^\natural\circ\SS\|_{L^\infty(Q_{R_0}(z_0))}\\
&\lesssim \lvv^{1/2}\|f\|_{L^2(\U)} +\|s\|_{L^\infty(\U)} 
+\|g\|_{L^2(\U\cap\Ge,\;\!\dif\mu)} +\|g\|_{L^\infty(\U\cap\Ge)}, 
\end{aligned} 
\end{align}
where we use the fact that $F^\natural\circ\SS$ coincides with the original solution $f$ in $\U$ owing to the uniqueness result given in Corollary~\ref{unique}. 

Recalling that $R_0\approx\lvv^{-1}$ and combining \eqref{bddF} with Lemma~\ref{bdd} applied to $f$ in the interior region $\{z\in\OO_T: \dist(z,\Sigma_T\cup(\{0\}\times\OO))\ge R_0\}$, we obtain the estimate~\eqref{bddf} as claimed. 

\subsubsection*{\textnormal{\textit{Step 4. Zooming in and oscillation decay}}} \label{step4}\ \\
Let us set $\overline{z}_0:=\SS(z_0)=\big(t_0,P^{-1}(x_0),P'^{-1}(x_0)v_0\big)$ for $z_0\in\Sigma_T$, and pick the constant $r_0\in(0,1]$ such that $r_0\approx\lvv^{-2}$ and $Q_{r_0}(\overline{z}_0)\subset\SS(Q_{R_0}(z_0))$, where we recall that $Q_{2R_0}(z_0)\subset(-1,T]\times\Q$. Taking $\overline{z}:=\T_{\overline{z}_0,r}(\tilde{z})$ with
$\tilde{z}:=(\tilde{t},\tilde{y},\tilde{w})\in Q_1$ and fixed $r\in(0,r_0]$, and regarding to \eqref{localequation}, we deduce that the equation
\begin{align*}
\left(\p_{\tilde{t}}+\tilde{w}\cdot\nabla_{\tilde{y}}\right)\widetilde{F}
=\nabla_{\tilde{w}}\cdot\big(\widetilde{A}\nabla_{\tilde{w}}\widetilde{F}\big) +\widetilde{B}\cdot\nabla_{\tilde{w}}\widetilde{F} 
+\tilde{c}\widetilde{F} +\tilde{s}
\end{align*} 
holds in the defective region $Q_1\backslash\{\tilde{z}:\,y_d>0,\,w_d<0,\,{\rm or\;}t<0\}$, where we defined 
\begin{align*}
\widetilde{F}&:=F^\natural\circ\T_{\overline{z}_0,r},\\
\widetilde{A}&:=A^\natural\circ\T_{\overline{z}_0,r},\\
\widetilde{B}&:=rB^\natural\circ\T_{\overline{z}_0,r},\\
\tilde{c}&:=r^2c^\natural\circ\T_{\overline{z}_0,r},\\
\tilde{s}&:=r^2s^\natural\circ\T_{\overline{z}_0,r}.
\end{align*} 
Due to the choice of $r_0$, the functions $|\widetilde{B}|$, $|\tilde{c}|$, $|\tilde{s}|$ are bounded by a universal constant. 

For $r\in(0,r_0]$, we define 
\begin{align*}
M_r:=\sup\nolimits_{\{\tilde{z}\in Q_1:\,\T_{\overline{z}_0,r}(\tilde{z})\in\SS(\Ge)\}}\widetilde{F}.
\end{align*} 
After extending the function $\big(\widetilde{F}-M_r\big)_+$ by zero to the region $\{\tilde{z}\in Q_1:y_d>0,\,w_d<0,\,{\rm or\;}t<0\}$, and normalizing it through 
\begin{align*}
\tilde{s}':=\tilde{c}\widetilde{F}& +\tilde{s},\\
F_r:=\big(\sup\nolimits_{Q_1}\big(\widetilde{F}-M_r\big)_+ &+\lambda^{-1}\|\tilde{s}'\|_{L^\infty(Q_1)}\big)^{\!-1}\big(\widetilde{F}-M_r\big)_+,
\end{align*} 
it turns out that the function $F_r$ is valued in $[0,1]$ over $Q_1$ and satisfies
\begin{align}\label{rescale}
\left(\p_{\tilde{t}}+\tilde{w}\cdot\nabla_{\tilde{y}}\right)F_r
\le\nabla_{\tilde{w}}\cdot\big(\widetilde{A}\nabla_{\tilde{w}}F_r\big) +\widetilde{B}\cdot\nabla_{\tilde{w}}F_r +\lambda
{\quad\rm in\ } Q_1, 
\end{align}  
where the universal constant $\lambda\in(0,1)$ is provided in Lemma~\ref{measure}. 

For $(x_0,v_0)\in\Gamma_0\cup\Gamma_-$, that is, $P^{-1}(x_0)\cdot e_d=0$ and $P'^{-1}(x_0)v_0\cdot e_d\le0$, we have 
\begin{align*}
y_d= r^3\tilde{y}\cdot e_d+ r^2\tilde{t}P'^{-1}(x_0)v_0\cdot e_d\ge0, 
\end{align*} 
whenever $\tilde{y}\cdot e_d\ge 0$ and $\tilde{t}\le0$. According to the definition of $F_r$, for any $\sigma,\tau\in(0,1)$ such that $Q_\sigma(-\tau,0,0)\subset Q_1$, we derive 
\begin{align*}
\left|\{F_r=0\}\cap Q_\sigma(-\tau,0,0)\right|\ge \frac{1}{4} |Q_\sigma(-\tau,0,0)|. 
\end{align*} 
Intuitively, $F_r$ is extended as a subsolution to the light gray area, and thus vanishes at least a quarter of $Q_1$; see Figure~\ref{img}. Applying Lemma~\ref{measure} with $\omega=1/4$ to the subsolution $F_r$ of \eqref{rescale} then yields that there exist some constants $\theta,\varrho\in(0,1)$ such that $F_r\le 1-\theta$ in $Q_\varrho$, which is recast as the decrease estimate of supremum that
\begin{align*}
\widetilde{F}-M_r\le (1-\theta)\sup\nolimits_{Q_1}\widetilde{F} -(1-\theta)\, M_r +\lambda^{-1}\|\tilde{s}'\|_{L^\infty(Q_1)}
{\quad\rm in\ }Q_\varrho. 
\end{align*}
Similarly as regards the setting $\big(m_r-\widetilde{F}\big)_+$ with $m_r:=\inf\nolimits_{\{\tilde{z}\in Q_1:\,\T_{\overline{z}_0,r}(\tilde{z})\in\SS(\Ge)\}}\widetilde{F}$, we have the increase estimate of infimum that
\begin{align*}
m_r-\widetilde{F}\le -(1-\theta)\inf\nolimits_{Q_1}\widetilde{F} +(1-\theta)\, m_r +\lambda^{-1}\|\tilde{s}'\|_{L^\infty(Q_1)}
{\quad\rm in\ } Q_\varrho. 
\end{align*}
Adding them together, we obtain the oscillation decay  
\begin{align*}
\begin{aligned}
\osc_{Q_\varrho}\widetilde{F}
&\le (1-\theta)\,\osc_{Q_1}\widetilde{F} +\theta\,(M_r-m_r) +2\lambda^{-1}\|\tilde{s}'\|_{L^\infty(Q_1)}\\
&\le (1-\theta)\,\osc_{Q_1}\widetilde{F} +\theta\,\osc_{\{\tilde{z}\in Q_1:\,\T_{\overline{z}_0,r}(\tilde{z})\in\SS(\Ge)\}}\widetilde{F}
+2\lambda^{-1}\|\tilde{s}'\|_{L^\infty(Q_1)}. 
\end{aligned}
\end{align*}
Rescaling back, it reads for any $r\in(0,r_0]$, 
\begin{align*}
\begin{aligned}
\osc_{Q_{\varrho r}(\overline{z}_0)}F^\natural
\le(1-\theta)\,\osc_{Q_r(\overline{z}_0)}F^\natural
+\theta\,\osc_{Q_r(\overline{z}_0)\cap\SS(\Ge)}g^\natural& \\
+2\lambda^{-1}r^2\big(\Lambda\|F^\natural\|_{L^\infty(Q_r(\overline{z}_0))} 
+\|s^\natural\|_{L^\infty(Q_r(\overline{z}_0))}\big)&.
\end{aligned}
\end{align*}
According to the standard iterative procedure (see for instance \cite[Section~8.10]{GT}), there is some universal constant $\alpha\in(0,1)$ such that for any $r\in(0,r_0]$, 
\begin{align}\label{osc-}
\begin{aligned}
\osc_{Q_{r}(\overline{z}_0)}F^\natural
&\lesssim  r_0^{-\alpha} r^{\alpha} \|F^\natural\|_{L^\infty(Q_{r_0}(\overline{z}_0))}
+r^{\alpha}\|s^\natural\|_{L^\infty(Q_{r_0}(\overline{z}_0))}
+\osc_{Q_{\sqrt{r_0r}}(\overline{z}_0)\cap\SS(\Ge)}\overline{g}\\
&\lesssim  r_0^{-\alpha} r^{\alpha} \|F^\natural\|_{L^\infty(Q_{r_0}(\overline{z}_0))}
+r^{\alpha}\|s\|_{L^\infty(\U)}
+r^{\beta/6}[\overline{g}]_{C^\beta(\W\cap\SS(\Ge))}. 
\end{aligned}
\end{align}

For $(x_0,v_0)\in\Gamma_+$, that is, $P^{-1}(x_0)\cdot e_d=0$ and $\overline{w}_d:=P'^{-1}(x_0)v_0\cdot e_d>0$, these boundary points for the solution $F^\natural$ of \eqref{localequation} reduce to the interior ones directly, for the reason that the interior estimate given by Lemma~\ref{interior} is applicable for solutions to \eqref{localequation} in the region $\W^\natural\cap\SS((0,T)\times\Q)$. More precisely, we have, for any $z_0\in\Sigma_T^+$ and $r\in(0,r_0/2]$, 
\begin{align}\label{osc+}
\osc_{Q_{r}(\overline{z}_0)}F^\natural
\lesssim \max\!\big\{r_0^{-\alpha},\overline{w}_d^{-\alpha}\big\}\;\! r^{\alpha} \|F^\natural\|_{L^\infty(Q_{r_0}(\overline{z}_0))}
+r^{\alpha}\|s\|_{L^\infty(\U)}. 
\end{align}
Applying \eqref{osc-} and \eqref{osc+} in the cases $r_0>\overline{w}_d$ and $r_0\le\overline{w}_d$, respectively, we see that the H\"older estimate \eqref{osc-} holds for any $\overline{z}_0=\SS(z_0)$ with $z_0\in\Sigma_T^+$. We remark that the estimate around the initial point $z_0\in\{0\}\times\OO$ also holds through the same zero extension argument; see \cite[Corollary 4.6]{YZ}. 

We now translate the H\"older estimates for $F^\natural$ into the ones for $f$. Let us abbreviate $B^-_r(z_0):=(t_0-r,t_0]\times B_r(x_0)\times B_r(v_0)$. Since the transformation $\SS$ and its inverse are bounded, for any $r\in(0,r_0]$, we have $B^-_{\nu r^3}(z_0)\subset Q_r(z_0)\subset B^-_{\nu^{-1} r}(z_0)$, for some universal constant $\nu\in(0,1)$. Recalling the definitions of $F^\natural$ and $\overline{F}$ with the fact \eqref{det}, we conclude that there is some constant $r_1\in(0,1]$ such that $r_1\approx r_0\approx\lvv^{-2}$, and for any $z_0\in\Sigma_T\cup(\{0\}\times\OO)$ and $r\in(0,r_1]$, 
\begin{align*}
\begin{aligned}
\osc_{\U\cap B^-_{r^3}(z_0)}f
\lesssim r_1^{-\alpha} r^\alpha\|F^\natural\circ\SS\|_{L^\infty(Q_{R_0}(z_0))} 
 +r^{\alpha}\|s\|_{L^\infty(\U)} +r^{\beta/6}[g]_{C^\beta(\U\cap\Ge)}. 
\end{aligned}
\end{align*}
Gathering this with \eqref{bddF} and picking the constant $\alpha\in(0,\beta/6]$ yield that 
\begin{align}\label{oscb}
\begin{aligned}
r^{-\alpha}\osc_{\U\cap B^-_{r^3}(z_0)}f
\lesssim \lvv^{1/2}r_1^{-\alpha} \|f\|_{L^\infty(\U)} +r_1^{-\alpha}\|s\|_{L^\infty(\U)}
+ r^{\beta/6-\alpha}[g]_{C^\beta(\U\cap\Ge)}. 
\end{aligned}
\end{align}
Together with the interior H\"older estimate in Lemma~\ref{interior}, we know that the weak solution $f$ is H\"older continuous in $\overline{\OO_T}$. Indeed, for any $z_0\in\{z\in\OO_T: \dist(z,\Sigma_T\cup(\{0\}\times\OO))\ge r_1\}$, 
\begin{align}\label{oscin}
r_1^{\alpha}[f]_{C^\alpha(B^-_{\nu r_1}(z_0))}
\lesssim \|f\|_{L^\infty(B^-_{r_1}(z_0))} +\|s\|_{L^\infty(B^-_{r_1}(z_0))}, 
\end{align} 
where we used the same notation of the (universal) constants $\alpha,\nu\in(0,1)$. Now that $r_1\approx\lvv^{-2}$, we conclude from the above two estimates that for any $z_0\in\overline{\OO_T}$ and $r\in(0,1]$,  
\begin{align*}
\begin{aligned}
\ [f]_{C^{\alpha/3}({\OO_T}\cap B^-_{1}(z_0))}
\lesssim \lvv^{1/2+2\alpha}\|f\|_{L^\infty({\OO_T}\cap B^-_{2}(z_0))} &\\
+\lvv^{2\alpha}\|s\|_{L^\infty({\OO_T}\cap B^-_{2}(z_0))}
+\lvv^{2\alpha-\beta/3} [g]_{C^\beta(\Ge\cap B^-_{2}(z_0))}&, 
\end{aligned}
\end{align*} 
which implies \eqref{holderf}. 

\subsubsection*{\textnormal{\textit{Step 5. Approximation}}}\label{step5}\ \\
We have to remove the additional assumptions used in the previous steps that the boundary data $g$ is continuously differentiable. To this end, we approximate $g$ by a sequence of smooth functions $\{g_j\}_{j\in\N}$, which preserves the same regularity as $g$ on $\Ge$. For each $j\in\N$, we acquire a continuous weak solution $f_j$ to \eqref{FP}. It follows that \eqref{oscb} and \eqref{oscin} hold for $f_j$, whose right hand sides are bounded independently of $j$. With the aid of the maximum principle given by Lemma~\ref{max}, we have 
\begin{align*}
\|f_i-f_j\|_{L^\infty(\OO_T)}
\lesssim \|g_i-g_j\|_{L^\infty(\Ge)} \rightarrow 0
{\quad\rm as\ }i,j\rightarrow\infty. 
\end{align*}
After passing to a subsequence, the passage $j\rightarrow\infty$ yields a bounded limiting function $f_\infty$ of $f_j$.  By the same argument as in the proof of Corollary~\ref{globalin}, we know that $f_\infty=f$ is the unique weak solution to \eqref{FP} in $\OO_T$ associated with $f=g$ on $\Ge$. In particular, when $g$ is continuous on $\Ge$, $f$ is globally continuous over $\overline{\OO_T}$. The proof is now complete. 
\end{proof}

\subsection{Global estimates}
\begin{lemma}\label{poly}
Let the constants $p\ge1$, $q\ge0$, and the function $f$ be a bounded weak solution to \eqref{FP} in $\OO_T$ with $\lv^qs\in L^2(\OO_T)$ and $\lv^qf|_{\Ge}\in L^2(\Ge,\dif\mu)$. Then, 
\begin{align*}
\begin{aligned}
\|\lv^qf^p\|_{C^0([0,T];L^2(\OO))} +\|\lv^q\nabla_v(f^p)\|_{L^2(\OO_T)}  +\|\lv^q f^p\|_{L^2(\p\OO_T,\;\!\dif\mu)}&\\
\lesssim \|\lv^qs^p\|_{L^2(\OO_T)} +\|\lv^qf^p\|_{L^2(\Ge,\;\!\dif\mu)}&. 
\end{aligned}
\end{align*} 
\end{lemma}

\begin{proof}
It is straightforward to check that $F:=f\phi$ and $S:=s\phi$ with $\phi=\phi(v)$ verifies  
\begin{align*}
(\p_t+v\cdot\nabla_x)F = \nabla_v\cdot(A\nabla_vF) +B\cdot\nabla_v F +cF +\nabla_v\cdot G_1 + G_0 
{\quad\rm in\ }\OO_T, 
\end{align*} 
provided that $f$ is a weak solution to \eqref{FP} in $\OO_T$, where $G_1$ and $G_0$ are given by 
\begin{align*}
G_1&:= -Af\nabla_v\phi,\\
G_0&:= -(A\nabla_vf+Bf)\cdot\nabla_v\phi +S. 
\end{align*} 
Let $\phi:=\lv^q$. By noticing that $\nabla_v\phi=\frac{qv}{\lv^2}\phi$, we acquire 
\begin{align*}
&f\nabla_v\phi=\frac{qv}{\lv^2}F,\\
&\frac{qv}{\lv^2}\otimes\nabla_vF=\nabla_v\phi\otimes\nabla_vf+\frac{qv}{\lv^2}\otimes \frac{qv}{\lv^2}F, 
\end{align*}
so that $G_1$ and $G_0$ are recast as 
\begin{align*}
G_1&=-\frac{qAv}{\lv^2}F,\\
G_0&= -\frac{qAv}{\lv^2}\cdot\nabla_vF +\frac{q^2Av\cdot v}{\lv^4}F 
-\frac{qB\cdot v}{\lv^2}F +S. 
\end{align*} 
According to the uniqueness of weak solutions, it is equivalent to consider the equation 
\begin{align}\label{FPq}
(\p_t+v\cdot\nabla_x)F = \nabla_v\cdot(A\nabla_vF+B'F) +(B+B')\cdot\nabla_v F +c'F +S {\quad\rm in\ }\OO_T, 
\end{align} 
where the new coefficients $B'$ and $c'$, defined by 
\begin{align*}
B'&:=-\frac{qAv}{\lv^2},\\
c'&:=c+\frac{q^2Av\cdot v}{\lv^4} -\frac{qB\cdot v}{\lv^2}, 
\end{align*} 
are bounded by a universal constant. By a similar version of Corollary~\ref{globalin}, we have 
\begin{align}\label{FPp}
\begin{aligned}
\|F^p\|_{C^0([0,T];L^2(\{t\}\times\OO))} +\|\nabla_v(F^p)\|_{L^2(\OO_T)} +\|F^p\|_{L^{2}(\p\OO_T,\;\!\dif\mu)} &\\
\lesssim \|S^p\|_{L^{2}(\OO_T)} +\|F^p\|_{L^{2}(\Ge,\;\!\dif\mu)} &,
\end{aligned}
\end{align} 
Indeed, taking $\chi(\iota)=\iota^{2p}$ and $\varphi=1$ in the renormalization formula for \eqref{FPq} (see an analogue in Lemma~\ref{trace}) yields that 
\begin{align*}
\int_{\{t\}\times\OO}F^{2p} -\int_{\{0\}\times\OO}F^{2p} +\int_{\Sigma_t}n_x\cdot vF^{2p}
+2p(2p-1)\int_{\OO_t}AF^{2p-2}\nabla_vF\cdot\nabla_vF &\\
= 2p\int_{\OO_t}\left[(B+2B'-2pB')F^{2p-1}\nabla_vF +c'F^{2p} +SF^{2p-1}\right] &. 
\end{align*} 
The claim then follows from the Cauchy–Schwarz inequality and Gr\"onwall's inequality. Replacing $q$ by $q/p$ in \eqref{FPp} implies the desired result. 
\end{proof}

The following global estimates are direct consequences of Proposition~\ref{local}, with Remark~\ref{localremark}, and Lemma~\ref{poly}. 

\begin{proposition}\label{global}
Let the constants $q\ge 2d$, $\beta\in(0,1]$, and the functions $s$ and $g$ satisfy $\lv^qs\in L^2(\OO_T)$, $\lv^{q-2d}s\in L^\infty(\OO_T)$, $\lv^qg\in L^2(\Ge,\dif\mu)$, $\lv^{q-2d}g\in L^\infty(\Ge)$. Then, there exists a unique weak solution $f$ to \eqref{FP} in $\OO_T$ associated with $f=g$ on $\Ge$, which satisfies 
\begin{align*}
\begin{aligned}
\big\|\lv^{q-2d}f\big\|_{L^\infty(\OO_T)}  
\lesssim \|\lv^{q}s\|_{L^2(\OO_T)} + \big\|\lv^{q-2d}s\big\|_{L^\infty(\OO_T)}&\\ 
+ \|\lv^{q}g\|_{L^2(\Ge,\;\!\dif\mu)} + \big\|\lv^{q-2d}g\big\|_{L^\infty(\Ge)}&.
\end{aligned} 
\end{align*}
If additionally $q>1/2+2d$ and $g\in C^\beta(\Ge)$, then there is some universal constant $\alpha\in(0,1)$ such that 
\begin{align*}
\begin{aligned}
\big\|\lv^{q-2d}f\big\|_{L^\infty(\OO_T)} +[f]_{C^\alpha(\OO_T)}
\lesssim \|\lv^{q}s\|_{L^2(\OO_T)} + \big\|\lv^{q-2d}s\big\|_{L^\infty(\OO_T)} &\\
+\|\lv^{q}g\|_{L^2(\Ge,\;\!\dif\mu)}  + \big\|\lv^{q-2d}g\big\|_{L^\infty(\Ge)}
+ [g]_{C^\beta(\Ge)}&. 
\end{aligned}
\end{align*} 
\end{proposition}

\section{Nonlocal reflection boundary problems}\label{diffuse}
This section is devoted to the regularity for solutions to nonlocal reflection boundary problems of \eqref{FP}. Let us recall the reflection operator 
\begin{align*}
\NN f(t,x,v):=\M(t,x,v)\int_{\R^d}f(t,x,v')\;\!(n_x\cdot v')_+\dif v'{\quad\rm in\ }\Sigma_T^-,
\end{align*} 
for $\M$ satisfying \eqref{M}. The proof is patterned after the argument from the previous section. We always assume that $\Omega$ is a bounded $C^{1,1}$-domain in $\R^d$. 

\subsection{A priori estimates}
Due to the same regularization procedure as the one used to deal with the inflow injection case, it is actually sufficient to prove the following lemma about the H\"older regularity of the macroscopic boundary quantity. 

We notice that the solution to \eqref{FP} with reflection boundary conditions lacks prescribed boundary data. We can first only get regularity around the outgoing boundary $\Gamma_+$, which exhibits singular estimates up to the grazing set $\Gamma_0$; for instance, the right hand side of \eqref{wbd} below will approach infinity as $r_0\rightarrow 0$. Fortunately, the weight $n_x\cdot v$ appearing in the nonlocal reflection operator $\NN$ will contribute to reconciling these estimates. This observation leads to the following lemma, which shows the regularity of $\NN f$ for $f$ solving \eqref{FP}. 

\begin{lemma}\label{diffusedata}
Let the constants $p=2+4d$, $q>1+d$, and the functions $s,\fin$ satisfy $\lv^qs\in L^\infty(\OO_T)$, $\lv^q\fin\in L^\infty(\OO)$. Then, for any weak solution $f$ to \eqref{FP} in $\OO_T$ such that $f|_{t=0}=\fin$ in $\OO$ and $\lv^{q+1/2} f\in L^p(\OO_T)$, the quantity $\Upsilon=\Upsilon[f]$, defined by 
\begin{align*}
\Upsilon[f](t,x):=\int_{\R^d}f(t,x,v)\;\!(n_x\cdot v)_+\dif v, 
\quad (t,x)\in [0,T]\times\p\Omega,
\end{align*}
satisfies that for any $\tau\in(0,T]$ with $\I_t^\tau:=[\max\{0,\tau-t\},\tau]$, we have 
\begin{align}\label{databdd}
\begin{aligned}
\|\Upsilon\|_{L^\infty(\I_1^\tau\times\p\Omega)}
\lesssim \big\|\lv^{q+1} f\big\|_{L^p(\I_2^\tau\times\OO)}
+ \|\lv^qs\|_{L^\infty(\I_2^\tau\times\OO)}&\\
+ \|\lv^qf\|_{L^\infty(\{t=0\}\cap(\I_2^\tau\times\OO))}&; 
\end{aligned}
\end{align} 
if additionally $\fin\in C^\beta(\OO)$, then there is some universal constant $\alpha\in(0,1)$ such that 
\begin{align}\label{dataholder}
\begin{aligned}
\ [\Upsilon]_{C^\alpha(\I_1^\tau\times\p\Omega)}
\lesssim \|\lv^{q} f\|_{L^\infty(\I_2^\tau\times\OO)}
+ \|\lv^qs\|_{L^\infty(\I_2^\tau\times\OO)} 
+[f]_{C^\beta(\{t=0\}\cap(\I_2^\tau\times\OO))}&. 
\end{aligned}
\end{align} 
\end{lemma}

\begin{proof}
We use the notation for the boundary flattening introduced in Lemma~\ref{flattening}. Recall that the diffeomorphism $P:(-R,R)^{d-1}\times(-R,0]\rightarrow U\cap\overline{\Omega}$, $y\mapsto x$, with the constant $R\in(0,1]$, is defined in \eqref{flatten}; the transformation $\SS:(t,x,v)\mapsto(t,y,w)$ is defined in \eqref{transformation}. Let us set 
\begin{align*}
\overline{f}:=f\circ\SS^{-1}
{\quad\rm in\ } \overline{\OO_T}. 
\end{align*}
For $\xi=(t,\check{y},0)\in [0,T]\times(-R,R)^{d-1}\times\{0\}$, the quantity $\Upsilon$ can be expressed as 
\begin{align*}
\Upsilon(t,P(\check{y},0))
=|(D\psi(\check{y}),1)| \int_{\{w_d>0\}} \overline{f}(\xi,w)\, w_d \dif w. 
\end{align*}

We thus have to acquire the regularity of the integrand above, which is expected due to the presence of the weight $w_d$. For fixed constant $\tau\in(0,T]$, we abbreviate the sets
\begin{align*}
\Z&:=\I_2^\tau\times(U\cap\Omega)\times\R^d,\\
\Xi&:=\I_1^\tau\times(-R/4,R/4)^{d-1}\times\{0\}, \\
\Pi&:=\I_{3/2}^\tau\times(-R/2,R/2)^{d-1}\times(-R/2,0]\times\R^d. 
\end{align*}
For $y_d=0$ and $w_d\ge1$, we are away from the grazing set. By using the facts that $q>1+d$ and the boundedness of $\SS$ and $\SS^{-1}$, and applying Proposition~\ref{local} with $p=2+4d$, we deduce that 
\begin{align}\label{w10}
\begin{aligned}
\sup_{\xi\in\Xi}\int_{\{w_d\ge1\}} \big|\overline{f}(\xi,w)\big|\, w_d \dif w 
\lesssim \sup_{(\xi,w)\in\Pi,\;\! w_d\ge1}\big|\la w\ra^q\overline{f}(\xi,w)\big|&\\
\lesssim \big\|\lv^{q+1}f\big\|_{L^p(\Z)} 
+ \|\lv^{q}s\|_{L^\infty(\Z)} 
+ \|\lv^{q}f\|_{L^\infty(\{t=0\}\cap\Z)}&;  
\end{aligned}
\end{align}
moreover, there is some universal constant $\alpha\in(0,1)$ such that 
\begin{align}\label{w11}
\begin{aligned}
\sup_{\xi\neq\xi'\in\Xi}\int_{\{w_d\ge1\}} |\xi-\xi'|^{-\alpha}\big|\overline{f}(\xi,w)-\overline{f}(\xi',w)\big| w_d \dif w
\lesssim \|\lv^{q}f\|_{L^\infty(\Z)} &\\
+ \|\lv^{q}s\|_{L^\infty(\Z)} + [f]_{C^\beta(\{t=0\}\cap\Z)}&. 
\end{aligned}
\end{align}

It now suffices to derive the regularity of $\overline{f}(\xi,w)$, with the weight $w_d$, near the grazing set, that is, $y_d=0$ and $w_d\in(0,1)$. To this end, we consider an arbitrary fixed point $\hat{z}:=(\hat{\xi},\hat{w})\in\Xi\times\R^d$ with $\hat{w}_d:=\hat{w}\cdot e_d\in(0,1)$. Based on the proof of Proposition~\ref{local} (see the derivation of the estimates \eqref{bddF}, \eqref{osc+} and \eqref{oscb}) with the dilation argument, for $r_0:=\hat{w}_d/(2\la\hat{w}\ra^2)$, we know the local boundedness estimate from $L^p$ to $L^\infty$ with $p=2+4d$ that 
\begin{align}\label{wbd}
\begin{aligned}
\big\|\overline{f}\big\|_{L^\infty(\Pi\cap Q_{r_0}(\hat{z}))} 
\lesssim \la\hat{w}\ra^{1/2}r_0^{-1}\big\|\overline{f}\big\|_{L^p(\Pi\cap Q_{2r_0}(\hat{z}))}  &\\
+ \big\|s\circ\SS^{-1}\big\|_{L^\infty(\Pi\cap Q_{2r_0}(\hat{z}))} 
+ \big\|f\circ\SS^{-1}\big\|_{L^\infty(\{t=0\}\cap\Pi\cap Q_{2r_0}(\hat{z}))} &.
\end{aligned}
\end{align}
and there is some universal constant $\alpha\in(0,1/3)$ such that for any $r\in(0,r_0]$, 
\begin{align}\label{wholder}
\begin{aligned}
{r}^{-3\alpha}\mathop{\osc}\limits_{(\xi,w)\in\Pi\cap Q_{{r}}(\hat{z})}\overline{f}(\xi,w)
\lesssim \la\hat{w}\ra^{1/2}r_1^{-3\alpha}\big\|\overline{f}\big\|_{L^\infty(\Pi\cap Q_{2{r}}(\hat{z}))}  &\\
+ r_1^{-3\alpha} \big\|s\circ\SS^{-1}\big\|_{L^\infty(\Pi\cap Q_{2{r}}(\hat{z}))} 
+ \big[f\circ\SS^{-1}\big]_{C^\beta(\{t=0\}\cap\Pi\cap Q_{2{r}}(\hat{z}))} &. 
\end{aligned}
\end{align}
It turns out from \eqref{wbd} and the arbitrariness of $\hat{z}\in\Xi\times\big\{\hat{w}\in\R^d:\hat{w}_d\in(0,1)\big\}$ that 
\begin{align*}
\begin{aligned}
\sup_{\xi\in\Xi}\int_{\{\hat{w}_d\in(0,1)\}} \big|\overline{f}(\xi,w)\big|\, w_d \dif w 
\lesssim \sup_{(\xi,w)\in\Pi,\;\!w_d\in(0,1)} \big|\la w\ra^{q-2}\overline{f}(\xi,w)\;\!w_d\big| &\\
\lesssim \big\|\lv^{q+1/2}f\big\|_{L^p(\Z)}
+\big\|\lv^{q-2}s\big\|_{L^\infty(\Z)} +\big\|\lv^{q-2}f\big\|_{L^\infty(\{t=0\}\cap\Z)}&.   
\end{aligned}
\end{align*}
Combining this with \eqref{w10} then yields the boundedness estimate \eqref{databdd} for $\Upsilon$. 

To obtain the H\"older estimate for $\Upsilon$, we apply \eqref{wholder} so that for any $w\in\R^d$ with $w_d\in(0,1)$ and $r_1:=w_d/(2\la w\ra^2)$, and any $\xi,\xi'\in\Xi$ such that $0<|\xi-\xi'|\le \nu r_1^3$,  
\begin{align*}
\begin{aligned}
\la w\ra^{q-2}|\xi-\xi'|^{-\alpha}\big|\overline{f}(\xi,w)-\overline{f}(\xi',w)\big|w_d 
\lesssim \sup_{(\xi,w)\in\Pi,\;\!w_d\in(0,2)} \big|\la w\ra^{q-3/2+6\alpha}\overline{f}(\xi,w)\;\!w_d^{1-3\alpha}\big| &\\
+\big\|\lv^{q-2+6\alpha}s\big\|_{L^\infty(\Z)}
+ [f]_{C^\beta(\{t=0\}\cap\Z)}&, 
\end{aligned}
\end{align*} 
where the constant $\nu\in(0,1)$ is universal. In addition, it is straightforward to check that for any $w_d\in(0,1)$ and $|\xi-\xi'|>\nu r_1^3$, 
\begin{align*}
\begin{aligned}
\la w\ra^{q-2}|\xi-\xi'|^{-\alpha}\big|\overline{f}(\xi,w)-\overline{f}(\xi',w)\big|w_d 
\lesssim \sup_{(\xi,w)\in\Pi,\;\!w_d\in(0,1)} \big|\la w\ra^{q-2+6\alpha}\overline{f}(\xi,w)\;\!w_d^{1-3\alpha}\big| . 
\end{aligned}
\end{align*}
Gathering the above two estimates and supposing $\alpha\le1/6$, we have 
\begin{align*}
\begin{aligned}
\sup_{\xi\neq\xi'\in\Xi}\int_{\{\hat{w}_d\in(0,1)\}} 
|\xi-\xi'|^{-\alpha}\big|\overline{f}(\xi,w)-\overline{f}(\xi',w)\big| w_d \dif w&\\
\lesssim \big\|\lv^{q-1/2}f\big\|_{L^\infty(\Z)} 
+ \big\|\lv^{q-1}s\big\|_{L^\infty(\Z)} + [f]_{C^\beta(\{t=0\}\cap\Z)}&. 
\end{aligned}
\end{align*}
Together with \eqref{w11}, we arrive at \eqref{dataholder}. This concludes the proof. 
\end{proof}

The following local-in-time estimate is an immediate consequence of Proposition~\ref{local} and Lemma~\ref{diffusedata}. Indeed, assisted by Lemma~\ref{diffusedata}, we see that the nonlocal reflection boundary problem is essentially the same as the inflow case. 

\begin{proposition}\label{localr}
Assume that the constants $q>1+d$, $\beta\in(0,1]$, $m\ge0$, and the functions $s,\fin$ satisfy $\lv^{m+q}s\in L^\infty(\OO_T)$, $\lv^{m+q}\fin\in L^\infty(\OO)$. Let $f$ be a weak solution to \eqref{FP} in $\OO_T$ associated with $f|_{t=0}=\fin$ in $\OO$ and $\gamma_-f=\NN  f$ in $\Sigma_T^-$, and $\lv^{m+(1+2d)(q+1)} f\in L^2(\OO_T)$. Then, for any $\tau\in(0,T]$ with $\I_t^\tau:=[\max\{0,\tau-t\},\tau]$, 
\begin{align}\label{bddr}
\begin{aligned}
\big\|\lv^mf\big\|_{L^\infty(\I_1^\tau\times\OO)}
\lesssim \big\|\lv^{m+(1+2d)(q+1)} f\big\|_{L^2(\I_2^\tau\times\OO)}&\\ 
+ \big\|\lv^{m+q}s\big\|_{L^\infty(\I_2^\tau\times\OO)}
+ \big\|\lv^{m+q}f\big\|_{L^\infty(\{t=0\}\cap(\I_2^\tau\times\OO))}&; 
\end{aligned} 
\end{align}
if additionally $\fin\in C^\beta(\OO)$ satisfies the compatibility condition $\gamma_-\fin=\gamma_-\NN\fin$, then there is some universal constant $\alpha\in(0,1)$ such that 
\begin{align}\label{holderr}
\begin{aligned}
\ [f]_{C^\alpha(\I_1^\tau\times\OO)}
\lesssim \|\lv^{q} f\|_{L^\infty(\I_2^\tau\times\OO)}
+ \|\lv^qs\|_{L^\infty(\I_2^\tau\times\OO)} 
+[f]_{C^\beta(\{t=0\}\cap(\I_2^\tau\times\OO))}&. 
\end{aligned}
\end{align} 
\end{proposition}

\begin{proof}
The combination of Proposition~\ref{local} and Lemma~\ref{diffusedata} implies that for the constant $p:=2+4d$ and for any $\tau\in(0,T]$, we have 
\begin{align*}
\begin{aligned}
\big\|\lv^{m}f\big\|_{L^\infty(\I_1^\tau\times\OO)}
\lesssim \big\|\lv^{m+q+1} f\big\|_{L^p(\I_2^\tau\times\OO)}&\\
+ \|\lv^{m+q}s\|_{L^\infty(\I_2^\tau\times\OO)}
+ \|\lv^{m+q}f\|_{L^\infty(\{t=0\}\cap(\I_2^\tau\times\OO))}&. 
\end{aligned} 
\end{align*}
To derive the estimate from $L^2$ to $L^\infty$, we write 
\begin{align*}
\big\|\lv^{m+q+1} f\big\|_{L^p(\I_2^\tau\times\OO)}
\le \big\|\lv^{m}f\big\|_{L^\infty(\I_2^\tau\times\OO)}^{(p-2)/p}
\big\|\lv^{m+(q+1)p/2}f\big\|_{L^2(\I_2^\tau\times\OO)}^{2/p} &\\
\lesssim \epsilon\|\lv^mf\|_{L^\infty(\I_2^\tau\times\OO)} 
+ \epsilon^{-2/(p-2)} \big\|\lv^{m+(q+1)p/2}f\big\|_{L^2(\I_2^\tau\times\OO)} &.
\end{align*}
By choosing $\epsilon>0$ sufficiently small, we obtain for some universal constant $C>0$, 
\begin{align*}
\begin{aligned}
\big\|\lv^mf\big\|_{L^\infty(\I_1^\tau\times\OO)}
\le \frac{1}{2}\|\lv^mf\|_{L^\infty(\I_2^\tau\times\OO)} 
+ C\big\|\lv^{m+(q+1)p/2}f\big\|_{L^2(\I_2^\tau\times\OO)} &\\
+ C\big\|\lv^{m+q}s\big\|_{L^\infty(\I_2^\tau\times\OO)}
+ C\big\|\lv^{m+q}f\big\|_{L^\infty(\{t=0\}\cap(\I_2^\tau\times\OO))} &. 
\end{aligned} 
\end{align*}
Notice that $\I_2^\tau=\I_1^\tau$ for $\tau\in(0,1]$, and $\I_2^\tau=\I_1^\tau\cup\I_1^{\tau-1}$ for $\tau\in[1,T]$. We thus conclude that for any $\tau\in(0,T]$, 
\begin{align*}
\begin{aligned}
\big\|\lv^mf\big\|_{L^\infty(\I_1^\tau\times\OO)}
\lesssim \big\|\lv^{m+(q+1)p/2}f\big\|_{L^2(\I_2^\tau\times\OO)} &\\
+ \big\|\lv^{m+q}s\big\|_{L^\infty(\I_2^\tau\times\OO)}
+ \big\|\lv^{m+q}f\big\|_{L^\infty(\{t=0\}\cap(\I_2^\tau\times\OO))}&, 
\end{aligned} 
\end{align*}
which is exactly \eqref{bddr} as asserted. 

Besides, the H\"older estimate~\eqref{holderr} is given by Proposition~\ref{local} and Lemma~\ref{diffusedata} directly. We point out that the compatibility condition $\gamma_-\fin=\gamma_-\NN\fin$ plays a role in verifying the H\"older continuity of $f|_{\Ge}$ near the initial time. 
\end{proof}

\subsection{Well-posedness result}
\begin{proposition}\label{globalr}
Let the constants $q>d+2$, $\beta\in(0,1]$, $m\ge q$, $l>m+(1+2d)q+d/2$, and the functions $s,\fin$ satisfy $\lv^ls\in L^\infty(\OO_T)$, $\lv^l\fin\in L^\infty(\OO)$. Then, there exists a unique bounded weak solution $f$ to \eqref{FP} in $\OO_T$ associated with $f|_{t=0}=\fin$ in $\OO$ and $\gamma_-f=\NN f$ in $\Sigma_T^-$, which satisfies 
\begin{align}\label{bddgr}
\begin{aligned}
\|\lv^mf\|_{L^\infty(\OO_T)} +\|\lv^{m+(1+2d)q} f\|_{C^0([0,T];L^2(\OO))}\\
\lesssim \|\lv^ls\|_{L^\infty(\OO_T)} + \|\lv^l\fin\|_{L^\infty(\OO)} &; 
\end{aligned} 
\end{align}
if additionally $\fin\in C^\beta(\OO)$ satisfies $\gamma_-\fin=\gamma_-\NN\fin$, then there is some universal constant $\alpha\in(0,1)$ such that 
\begin{align}\label{holdergr}
\begin{aligned}
\ [f]_{C^\alpha(\OO_T)}
\lesssim \|\lv^ls\|_{L^\infty(\OO_T)} + \|\lv^l\fin\|_{L^\infty(\OO)} +[\fin]_{C^\beta(\OO)}&. 
\end{aligned}
\end{align} 
\end{proposition}

\begin{remark}
The range of the constants $q,l$ above are not optimal. 
\end{remark}

\begin{proof}
The proof is based on the following iterative scheme of inflow boundary problems. By Corollary~\ref{globalin}, we assume that $f_n$ solves \eqref{FP} in $I_k\times\OO$, for the time interval $I_k:=[k\tau,(k+1)\tau]\cap[0,T]$ with $k\in\N$, 
\begin{align*}
f_n|_{t=0}=\fin{\quad\rm and\quad}
\gamma_-f_{n+1}=\NN f_{n}{\quad\rm for\ }n\in\N, \quad
\gamma_-f_0=0, 
\end{align*}
where the constant $\tau\in(0,1]$ is to be determined. Let us abbreviate $\OO^k:=I_k\times\OO$. In addition, we set $I_{-1}:=\{0\}$ and $\OO^{-1}:=\{0\}\times\OO$ for convenience. 

We first have to establish a priori estimates for $f_n$. For the constants $p:=2+4d$, $q>2+d$, $q_1>q+d/p$, applying Lemma~\ref{poly} implies that 
\begin{align*}
\begin{aligned}
\|\lv^q f_{n+1}\|_{C^0(I_k;L^p(\OO))} 
+ \|\lv^q f_{n+1}\|_{L^p(I_k\times\Gamma,|n_x\cdot v|)} &\\
\lesssim \|\lv^{q_1}s\|_{L^\infty(\OO^k)} 
+ \|\lv^{q}f_{n+1}\|_{L^p(\{k\tau\}\times\OO)}
+ \|\lv^{q}\NN f_{n}\|_{L^p(I_k\times\Gamma_-,|n_x\cdot v|)}&. 
\end{aligned}
\end{align*} 
By the definition of the operator $\NN$ and Lemma~\ref{diffusedata}, we have
\begin{align}\label{estb}
\begin{aligned}
\|\lv^{q}\NN f_{n}\|_{L^p(I_k\times\Gamma_-,|n_x\cdot v|)}
\lesssim \tau^{1/p} \big\|\lv^{q_1+1/p}\M\big\|_{L^\infty(I_k\times\Gamma)} \|\Upsilon[f_{n}]\|_{L^\infty(I_k\times\p\Omega)} &\\
\lesssim \tau^{1/p} \big\|\lv^{q} f_{n}\big\|_{C^0(I_k\cup I_{k-1};L^p(\OO))}
+ \tau^{1/p} \|\lv^{q}s\|_{L^\infty((I_k\cup I_{k-1})\times\OO)} &\\
+ \tau^{1/p} \|\lv^{q}f_{n}\|_{L^\infty(\{t=0\}\cap\OO^{k-1})}&. 
\end{aligned}
\end{align}
Combining the above two estimates and choosing $\tau$ sufficiently small yields that for some universal constant $C>0$, 
\begin{align}\label{estinb}
\begin{aligned}
\|\lv^q f_{n+1}\|_{C^0(I_k;L^p(\OO))}
+ \|\lv^q f_{n+1}\|_{L^p(I_k\times\Gamma,|n_x\cdot v|)} &\\
\le C\|\lv^{q_1}s\|_{L^\infty(\OO_T)} + C\|\lv^{q}f_{n+1}\|_{L^p(\{k\tau\}\times\OO)} 
+ \frac{1}{2}\|\lv^{q} f_{n}\|_{C^0(I_k;L^p(\OO))}&\\
+ \|\lv^{q} f_{n}\|_{C^0(I_{k-1};L^p(\OO))}
 + \|\lv^qf_{n}\|_{L^\infty(\{t=0\}\cap\OO^{k-1})} &. 
\end{aligned} 
\end{align}

In order to derive the convergence for $f_n$, we observe that the function $f_{n+1}-f_n$ also solves \eqref{FP} in $I_k\times\OO$ associated with $s=0$,  
\begin{align*}
(f_{n}-f_{n-1})|_{t=0}=0 {\quad\rm and\quad}
\gamma_-(f_{n+1}-f_n)=\NN (f_{n}-f_{n-1}){\quad\rm for\ }n\in\N_+. 
\end{align*}
It then follows from \eqref{estinb} that 
\begin{align}\label{cdn}
\begin{aligned}
\|\lv^{q}(f_{n+1}-f_{n})\|_{C^0(I_k;L^p(\OO))}  
+ \|\lv^q (f_{n+1}-f_{n})\|_{L^p(I_k\times\Gamma,|n_x\cdot v|)}&\\
\le C\|\lv^{q}(f_{n+1}-f_{n})\|_{L^p(\{k\tau\}\times\OO)}
+\frac{1}{2}\|\lv^{q} (f_{n}-f_{n-1})\|_{C^0(I_k;L^p(\OO))}  &\\
+ \|\lv^{q} (f_{n}-f_{n-1})\|_{C^0(I_{k-1};L^p(\OO))} &.
\end{aligned} 
\end{align}
Besides, with the aid of the energy estimate given by Lemma~\ref{poly} (with $p=1$ and $q=0$), as well as \eqref{estb}, we have  
\begin{align}\label{esten}
\begin{aligned}
\|\nabla_v(f_{n+1}-f_{n})\|_{L^2(\OO^k)} 
\!\lesssim \|\NN(f_{n+1}-f_{n})\|_{L^2(I_k\times\Gamma_-,|n_x\cdot v|)}\! +\|f_{n+1}-f_{n}\|_{L^2(\{k\tau\}\times\OO)}&\\
\lesssim \|\lv^q(f_{n}-f_{n-1})\|_{C^0(I_k;L^p(\OO))} +\|\lv^{q}(f_{n+1}-f_{n})\|_{L^p(\{k\tau\}\times\OO)}&. 
\end{aligned} 
\end{align}
In particular, from \eqref{cdn} and \eqref{esten} with $k=0$, we obtain by an iteration 
\begin{align}\label{cd0}
\begin{aligned}
\|\nabla_v(f_{n+1}-f_{n})\|_{L^2(\OO^0)} 
\lesssim \|\lv^{q}(f_{n+1}-f_{n})\|_{C^0(I_0;L^p(\OO))} &\\
\le\frac{1}{2}\|\lv^{q} (f_{n}-f_{n-1})\|_{C^0(I_0;L^p(\OO))} 
\le\frac{1}{2^n}\|\lv^{q} (f_{1}-f_{0})\|_{C^0(I_0;L^p(\OO))} &. 
\end{aligned} 
\end{align}
Thanks to \eqref{estinb}, the right hand side tends to zero as $n\rightarrow\infty$. We conclude the convergence of $f_n$ in $C^0(I_0;L^p(\OO,\lv^{pq}))$ and in $L^2(I_0\times\Omega;H^1(\R^d))$. We remark that the boundary identity $\gamma_-f=\NN f$ as the limit of $\gamma_-f_{n-1}=\NN f_n$ on $I_0\times\Gamma_-$ make sense, owing to \eqref{cdn}. To proceed the convergence of $f_n$ in $\OO^1$, we gather the estimates \eqref{cdn}, \eqref{esten}, \eqref{cd0} to see that  
\begin{align*}
\begin{aligned}
\|\nabla_v(f_{n+1}-f_{n})\|_{L^2(\OO^1)} 
\lesssim \|\lv^{q}(f_{n}-f_{n-1})\|_{C^0(I_1;L^p(\OO))} +\|\lv^{q}(f_{n+1}-f_{n})\|_{L^p(\{\tau\}\times\OO)}  &\\ 
\le \frac{1}{2}\|\lv^{q} (f_{n-1}-f_{n-2})\|_{C^0(I_1;L^p(\OO))}
+ \frac{C}{2^n}\|\lv^{q} (f_{1}-f_{0})\|_{C^0(I_{0};L^p(\OO))} &.
\end{aligned} 
\end{align*}
By iterating and sending $n\rightarrow\infty$, we acquire the solution $f$ to \eqref{FP} in $(I_1\cup I_0)\times\OO$. The global solution to \eqref{FP} in $\OO_T$ is then constructed by applying such arguments repeatedly. 

Next, by Lemma~\ref{poly} and Proposition~\ref{localr}, we get the global boundedness estimate   
\begin{align*}
\begin{aligned}
\|\lv^mf\|_{L^\infty(\OO_T)}
+\big\|\lv^{m+(1+2d)q}f\big\|_{C^0([0,T];L^2(\OO))}
\lesssim \big\|\lv^{m+(1+2d)q}s\big\|_{L^2(\OO_T)} &\\
+ \big\|\lv^{m+(1+2d)q}\fin\big\|_{L^2(\OO)}
+\|\lv^{m+q}s\|_{L^\infty(\OO_T)} + \|\lv^{m+q}\fin\|_{L^\infty(\OO)} &, 
\end{aligned} 
\end{align*}
which implies \eqref{bddgr} with the constant $l>m+(1+2d)q+d/2$. The global H\"older estimate~\eqref{holdergr} is also given by Proposition~\ref{localr}. 

Let us finally show the uniqueness for the weak solution $f$ to \eqref{FP} in $\OO_T$ with $s=\fin=0$ and $\gamma_-f=\NN f$. Using Lemma~\ref{poly} and Proposition~\ref{localr} yields that for any $t\in(0,T]$, 
\begin{align*}
\begin{aligned}
\big\|\lv^{(2+2d)q} f\big\|_{C^0([0,t];L^2(\OO))}
\lesssim \big\|\lv^{(2+2d)q}\NN f\big\|_{L^2(\Sigma_t^-,|n_x\cdot v|)}&\\
\lesssim \big\|\lv^{(2+2d)q+(1+d)/2}\M\big\|_{L^\infty(\Sigma_t)} \big\|\lv^{q}f\big\|_{L^\infty(\Sigma_t)} 
\lesssim  \big\|\lv^{(2+2d)q} f\big\|_{L^2(\OO_t)}&. 
\end{aligned}
\end{align*} 
By Gr\"onwall's inequality, we get $f=0$ in $\OO_T$ so that see the the uniqueness of solutions in the class $C^0([0,T];L^2(\OO,\lv^{(4+4d)q}))$. This finishes the proof. 
\end{proof}

\section{Specular reflection boundary problems}\label{specular}
Building upon the mirror extension technique developed in \cite{Nier}, \cite{GHJO}, we demonstrate the regularity for the specular reflection boundary problems, with the proof in a similar spirit to that presented in \cite{GHJO}. The solution to \eqref{FP} can be extended outside of the domain directly as a solution to a modified equation. The treatment of this type of boundary condition is thus simpler than the cases that we have already addressed. As in the previous two sections, we still assume that $\Omega$ is a bounded $C^{1,1}$-domain in $\R^d$. 

\begin{proposition}\label{locals}
Assume that the constants $\beta,\epsilon\in(0,1]$, and the functions $s\in L^\infty(\OO_T)$, $\fin\in L^\infty(\OO)$. Let $f$ be a weak solution to \eqref{FP} in $\OO_T$ associated with $f|_{t=0}=\fin$ in $\OO$ and $\gamma_-f=\RR f$ in $\Sigma_T^-$. Then, for any $z_0=(t_0,x_0,v_0)\in\overline{\OO_T}$, we have 
\begin{align*}
\begin{aligned}
\|f\|_{L^\infty({\OO_T}\cap B_1(z_0))} 
\lesssim \lvv^{1+2d}\|f\|_{L^2({\OO_T}\cap B_2(z_0))}\\ 
+\|s\|_{L^\infty({\OO_T}\cap B_2(z_0))} +\|f\|_{L^\infty(\{t=0\}\cap B_2(z_0))}&;  
\end{aligned} 
\end{align*}
if additionally $\fin\in C^\beta(\OO)$ satisfies the compatibility condition $\gamma_-\fin=\gamma_-\RR\fin$, then there is some universal constant $\alpha\in(0,1)$ such that 
\begin{align*}
\begin{aligned}
\ [f]_{C^\alpha({\OO_T}\cap B_{1}(z_0))}
\lesssim \lvv^{\epsilon}\|f\|_{L^\infty({\OO_T}\cap B_{2}(z_0))} &\\
+\lvv^{\epsilon}\|s\|_{L^\infty({\OO_T}\cap B_{2}(z_0))}
 + [f]_{C^\beta(\{t=0\}\cap B_{2}(z_0))}&. 
\end{aligned}
\end{align*} 
\end{proposition}

\begin{proof}
We first point out that the results of existence and uniqueness for the weak solution $f$ have been proved in Corollary~\ref{globalre}; moreover, we know that the trace $\gamma f$ is well-defined in $L^\infty(\Sigma_T)$. Let us now reduce the regularity estimate near $z_0\in\Sigma_T$ to the interior one by means of the mirror extension technique. Recall the coordinates $z=(t,x,v)\in[0,T]\times U\times B_1(v_0)$, and the transformation $\SS:z\mapsto\overline{z}=(t,y,w)$ defined in \eqref{transformation}; see also Lemma~\ref{flattening}. Let 
\begin{align*}
\U&:=[0,T]\times (U\cap\overline{\Omega})\times B_1(v_0),\\
\Y&:=\SS\left([0,T]\times U\times B_1(v_0)\right). 
\end{align*}
To extend the solution $f$ of $z\in\U$, we define the function $\widehat{f}$ of $\overline{z}\in\Y$ by 
\begin{align*}
&\widehat{f}:=\det\!\left(\frac{\p z}{\p\overline{z}}\right) f\circ\SS^{-1} {\quad\rm in\ }\Y\cap\{y_d\le0\},\\
&\widehat{f}:=\det\!\left(\frac{\p z}{\p\overline{z}}\right) f\circ\SS^{-1}\!\circ\Rm {\quad\rm in\ }\Y\cap\{y_d>0\},
\end{align*}
where we set the mirror reflection operator 
\begin{align*}
\Rm(t,y,w):=\left(t,\check{y},-y_d,\check{w},-w_d\right).
\end{align*}

Regarding to the boundary condition, by recalling \eqref{DP} and the identity $(D\mm)^T\nn=0$ whenever $y_d=0$, we have, for any $\check{y}\in(-R,R)^{d-1}$ and $w=(\check{w},w_d)\in \R^{d-1}\times\R$,
\begin{align*}
\begin{aligned}
P'\;\;\!\!\!\!(\check{y},0)(\check{w},-w_d)&= (D\mm;\nn)(\check{w},-w_d)^T = D\mm\,\check{w}-w_d\nn\\
&=\RR_x(D\mm\,\check{w}+w_d\nn) = \RR_x(P'\;\;\!\!\!\!(\check{y},0)w), 
\end{aligned}
\end{align*} 
where the specular reflection operator $\RR_x$ on $\p\Omega$ is defined by $\RR_xu:=u-2(n_x\cdot u)\;\!n_x$ for any $u\in\R^d$. Together with the boundary condition for $f$, we obtain  
\begin{align}\label{reflection}
\begin{aligned}
f\circ\SS^{-1}\!\circ\Rm(\overline{z})
&=f\left(t,P(\check{y},0),P'(\check{y},0)(\check{w},-w_d)\right)\\
&=f(t,x,\RR_x v)=f\left(t,x,v\right)=f\circ\SS^{-1}(\overline{z}), 
\end{aligned}
\end{align} 
which roughly means that $\widehat{f}$ is continuous across $\Y\cap\{y_d=0\}$.

We are now able to perform the same derivation as for \eqref{extendequation} in Step~\hyperref[step1]{1} of Subsection~\ref{inflowproof}. Indeed, it turns out that for any $\varphi\in C_c^1(\Y)$, 
\begin{align*}
\begin{aligned}
\int_{\Y\cap\{y_d<0\}\cap(\{t\}\times\OO)} \widehat{f}\, \varphi\;
-\int_{\Y\cap\{y_d<0\}\cap(\{0\}\times\OO)} \widehat{f}\, \varphi\;
+\int_{\Y\cap\{y_d=0\}} w_d\,\widehat{f}\, \varphi&\\
=\int_{\Y\cap\{y_d<0\}} \Big[\widehat{f}(\p_t+w\cdot\nabla_y) -\overline{A}\,\nabla_w\widehat{f}\cdot\nabla_w +\overline{B}\cdot\nabla_w \widehat{f} +\overline{c}\,\widehat{f} +\overline{s} \Big]\,\varphi&, 
\end{aligned}
\end{align*} 
where the coefficients $\overline{A},\overline{B},\overline{c},\overline{s}$ are defined in the region $\Y\cap\{y_d\le 0\}$ by  
\begin{align*}
\overline{A}&:=P'^{-1}(A\circ\SS^{-1}) P'^{-T},\\
\overline{B}&:=P'^{-1} B\circ\SS^{-1}\! -((D_xw)v) \circ\SS^{-1},\\
\overline{c}&:=c \circ\SS^{-1}\! - \nabla_w\cdot\left[((D_xw)v) \circ\SS^{-1}\right],\\
\overline{s}&:=\det\!\left(\frac{\p z}{\p\overline{z}}\right) s \circ\SS^{-1}. 
\end{align*} 
On account of this, applying change of variables twice with the relation that $\Rm\circ\Rm={\rm id}$, as well as using the boundary condition \eqref{reflection}, yields that
\begin{align*}
\int_{\Y\cap\{y_d>0\}\cap(\{t\}\times\OO)} \widehat{f}\, \varphi\;
-\int_{\Y\cap\{y_d>0\}\cap(\{0\}\times\OO)} \widehat{f}\, \varphi\;
-\int_{\Y\cap\{y_d=0\}} w_d\,\widehat{f} \varphi&\\
=\int_{\Y\cap\{y_d<0\}\cap(\{t\}\times\OO)} \widehat{f}\, \varphi\circ\Rm\,
-\int_{\Y\cap\{y_d<0\}\cap(\{0\}\times\OO)} \widehat{f}\, \varphi\circ\Rm\,
+\int_{\Y\cap\{y_d=0\}} w_d\,\widehat{f}\, \varphi\circ\Rm&\\
=\int_{\Y\cap\{y_d<0\}} \Big[\widehat{f}\,(\p_t+w\cdot\nabla_y)-\overline{A}\,\nabla_w\widehat{f}\cdot\nabla_w +\overline{B}\cdot\nabla_w \widehat{f} 
+\overline{c}\,\widehat{f} +\overline{s} \Big]\, \varphi\circ\Rm&\\
=\int_{\Y\cap\{y_d>0\}} \Big[\widehat{f}\,(\p_t+w\cdot\nabla_y) -\widehat{A}\,\nabla_w\widehat{f}\cdot\nabla_w +\widehat{B}\cdot\nabla_w \widehat{f} +\widehat{c}\,\widehat{f}+\widehat{s} \Big]\, \varphi&,
\end{align*} 
where with the notation of the $d\times d$ diagonal matrix $J:=\diag(1,\ldots,1,-1)$, we extended the coefficients $\overline{A},\overline{B},\overline{c},\overline{s}$ as $\widehat{A},\widehat{B},\widehat{c},\widehat{s}$ to the region $\Y\cap\{y_d> 0\}$ by setting   
\begin{align*}
\widehat{A}&:=J\big(\overline{A}\circ\Rm\big) J,\\
\widehat{B}&:=J\overline{B}\circ\Rm,\\
\widehat{c}&:=\overline{c}\circ\Rm,\\
\widehat{s}&:=\overline{s}\circ\Rm. 
\end{align*} 
In view of the above two formulations valid in $\Y\cap\{y_d<0\}$ and $\Y\cap\{y_d>0\}$, we conclude 
\begin{align}\label{sequation}
(\p_t+w\cdot\nabla_y)\widehat{f}
=\nabla_w\cdot\big(\widehat{A}\nabla_w\widehat{f}\;\!\big) +\widehat{B}\cdot\nabla_w\widehat{f}+ \widehat{c}\,\widehat{f}+\widehat{s}
{\quad\rm in\ }\Y. 
\end{align} 
By definition and Lemma~\ref{flattening}, it is readily checked that 
\begin{align*}
\begin{aligned}
\big\|((D_xw)v)\circ\SS^{-1}\big\|_{L^\infty(\Y)}
+\big\|\nabla_w\cdot\big[((D_xw)v)\circ\SS^{-1}\big]\big\|_{L^\infty(\Y)}
\lesssim \lvv^2. 
\end{aligned}
\end{align*}
We thus see that all the eigenvalues of $\widehat{A}$ lie in $[K^{-1},K]$, and $|\widehat{B}|, |\widehat{c}|$ are bounded by $K\lvv^2$, for some universal constant $K>1$. 

We next sketch the remaining part of the proof. As with the previous inflow problems, we have to take care of the coefficients of lower order terms. To derive the equation with coefficients bounded independent of $\lvv$, we consider the constant $r_0\approx\lvv^{-2}$ such that $Q_{2r_0}(z_0)\subset\U$. Then, the function $\widehat{F}:=\widehat{f}\circ\T_{\overline{z}_0,r_0}$ with $\overline{z}_0:=\SS(z_0)$ solves the same type equation as \eqref{sequation} in $Q_2$; see the same argument as in Step~\hyperref[step4]{4} of Subsection~\ref{inflowproof}. Using the interior estimate given by Lemma~\ref{interior} yields the H\"older estimate of $\widehat{f}$ in $B_{r_0}(z_0)$ for any $z_0\in\Sigma_T\cup(\{0\}\times\OO)$, which in turn gives the same regularity for $f$. Combining this with Lemma~\ref{interior} applied to $f$ in the interior region $\{z\in\OO_T: \dist(z,\Sigma_T\cup(\{0\}\times\OO))\ge r_0\}$, we obtain the H\"older estimate of $f$ in $\overline{\OO_T}$. We finally remark that the treatment around the initial point $z_0\in\{0\}\times\OO$ is the same as the one in Step~\hyperref[step4]{4} of Subsection~\ref{inflowproof}; see also \cite[Corollary 4.6]{YZ}. The $C^\beta$-H\"older regularity of $\widehat{f}|_{t=0}$ requires the conditions $\fin\in C^\beta(\OO)$ and $\gamma_-\fin=\gamma_-\RR\fin$. This completes the proof. 
\end{proof}

By the same derivations as Lemma~\ref{poly} and Proposition~\ref{global} for the inflow boundary problems, we have the following global estimates for specular reflection boundary problems. 
\begin{proposition}\label{globals}
Let the constants $q>0$, $\beta\in(0,1]$, and the functions $s$ and $\fin$ satisfy $\lv^qs\in L^2(\OO_T)$ and $\lv^q\fin\in L^2(\OO)$. Then, there exists a unique weak solution $f$ to \eqref{FP} in $\OO_T$ associated with $f|_{t=0}=\fin$ in $\OO$ and $\gamma_-f=\RR f$ in $\Sigma_T^-$, which satisfies 
\begin{align*}
\begin{aligned}
\|\lv^qf\|_{C^0([0,T];L^2(\OO))} +\|\lv^q\nabla_vf\|_{L^2(\OO_T)}
\lesssim \|\lv^qs\|_{L^2(\OO_T)} . 
\end{aligned}
\end{align*} 
Furthermore, if $q\ge 2d$, $\lv^{q-2d}s\in L^\infty(\OO_T)$ and $\lv^{q-2d}\fin\in L^\infty(\OO)$, then 
\begin{align*}
\begin{aligned}
\big\|\lv^{q-2d}f\big\|_{L^\infty(\OO_T)} 
\lesssim \|\lv^{q}s\|_{L^2(\OO_T)} + \big\|\lv^{q-2d}s\big\|_{L^\infty(\OO_T)} &\\
+ \|\lv^{q}\fin\|_{L^2(\OO)} + \big\|\lv^{q-2d}\fin\big\|_{L^\infty(\OO)}&;
\end{aligned} 
\end{align*}
and if $q>2d$ and $\fin\in C^\beta(\OO)$ with $\gamma_-\fin=\gamma_-\RR\fin$, then there is some universal constant $\alpha\in(0,1)$ such that 
\begin{align*}
\begin{aligned}
\big\|\lv^{q-2d}f\big\|_{L^\infty(\OO_T)} +[f]_{C^\alpha(\OO_T)}
\lesssim \|\lv^{q}s\|_{L^2(\OO_T)} + \big\|\lv^{q-2d}s\big\|_{L^\infty(\OO_T)} &\\
+\|\lv^{q}\fin\|_{L^2(\OO)} + \big\|\lv^{q-2d}\fin\big\|_{L^\infty(\OO)}
+ [\fin]_{C^\beta(\OO)}&. 
\end{aligned}
\end{align*} 
\end{proposition}

\appendix

\section{Optimality of the H\"older class}\label{counterexample}
We show the optimality of the H\"older class for classical solutions to \eqref{FP}  with constant coefficients and the absorbing boundary condition. Through the description in \cite[Claim~3.7]{HJV}, we make the construction here for the sake of completeness. 

Consider the steady Fokker-Planck equation 
\begin{align}\label{steady}
v\p_xf(x,v)=\p_v^2f(x,v),\quad (x,v)\in\R_+\times\R, 
\end{align}
with the absorbing boundary condition on $\{x=0,v>0\}$, that is, $f(0,v)=0$ for $v>0$. 

The goal is to construct the solution of the type $f(x,v)=x^\frac{1}{6}\Psi(\tau)$ with $\tau=-\frac{v^3}{9x}$, for the ansatz $\Psi:\R\rightarrow\R$. Plugging it into \eqref{steady} yields so-called Kummer's equation
\begin{align*}
\tau\Psi''(\tau)+\left(\frac{2}{3}-\tau\right)\Psi'(\tau)+\frac{1}{6}\Psi(\tau)=0, 
\end{align*}
whose two linearly independent solutions are $M\!\left(-\frac{1}{6},\frac{2}{3},\tau\right)$ and
$\tau^\frac{1}{3} M\!\left(\frac{1}{6},\frac{4}{3},\tau\right)$, for Kummer's function $M$; see \cite[Chapter~6]{Wangzhuxi}. Tricomi's function $\Psi(\tau)=\Psi\!\left(-\frac{1}{6},\frac{2}{3},\tau\right)$ is then given by their linear combination, 
\begin{align}\label{Tricomi}
\Psi(\tau):=\frac{\Upgamma\!\left(\frac{1}{3}\right)}{\Upgamma\!\left(\frac{1}{6}\right)} M\!\left(-\frac{1}{6},\frac{2}{3},\tau\right)
+ \frac{\Upgamma\!\left(-\frac{1}{3}\right)}{\Upgamma\!\left(-\frac{1}{6}\right)}\, \tau^\frac{1}{3} M\!\left(\frac{1}{6},\frac{4}{3},\tau\right), 
\end{align}
where $\Upgamma(\cdot)$ is the Gamma function. We have to use the asymptotic expansion of $M,\Psi$ when $|\tau|\rightarrow\infty$; see for instance \cite[Section~6.6]{Wangzhuxi}. On the one hand, for $\tau\rightarrow-\infty$,  
\begin{align*}
M\!\left(-\frac{1}{6},\frac{2}{3},\tau\right) 
= \frac{\Upgamma\!\left(\frac{2}{3}\right)}{\Upgamma\!\left(\frac{5}{6}\right)}
\,|\tau|^\frac{1}{6} \left[1+O\!\left(|\tau|^{-1}\right)\right], &\\
\tau^\frac{1}{3} M\!\left(\frac{1}{6},\frac{4}{3},\tau\right) 
=- \frac{\Upgamma\!\left(\frac{4}{3}\right)}{\Upgamma\!\left(\frac{7}{6}\right)}
\,|\tau|^{\frac{1}{6}} \left[1+O\!\left(|\tau|^{-1}\right)\right]. 
\end{align*}
Since $\Upgamma\!\left(\frac{1}{3}\right)\Upgamma\!\left(\frac{2}{3}\right)=-\Upgamma\!\left(-\frac{1}{3}\right)\Upgamma\!\left(\frac{4}{3}\right)$ and $\Upgamma\!\left(\frac{1}{6}\right)\Upgamma\!\left(\frac{5}{6}\right)=-\Upgamma\!\left(-\frac{1}{6}\right)\Upgamma\!\left(\frac{7}{6}\right)$, 
we see that the leading terms cancel out for $\tau\rightarrow-\infty$ so that $\Psi(\tau)=O\big(|\tau|^{-\frac{5}{6}}\big)$, which implies 
\begin{align*}
f(x,v)=O\big(xv^{-\frac{5}{2}}\big),{\quad\rm for\ } x^{-1}{v^3}\rightarrow\infty. 
\end{align*}
In particular, $f(0,v)=0$ for $v>0$. On the other hand, for $\tau\rightarrow\infty$, the asymptotic expansion of Tricomi's function $\Psi$ directly yields 
\begin{align*}
\Psi(\tau)=|\tau|^\frac{1}{6}\left[1+O\!\left(|\tau|^{-1}\right)\right],
\end{align*}
and hence
\begin{align*}
f(x,v)= 3^{-\frac{1}{3}}|v|^\frac{1}{2} +O\big(xv^{-\frac{5}{2}}\big),{\quad\rm for\ } x^{-1}{v^3}\rightarrow-\infty. 
\end{align*}
Besides, Kummer's functions $M\!\left(-\frac{1}{6},\frac{2}{3},\cdot\right)$ and $M\!\left(\frac{1}{6},\frac{4}{3},\cdot\right)$ are analytic on $\R$. It then follows from \eqref{Tricomi} that the solution $f(x,v)=x^\frac{1}{6}\Psi\big(\!-\frac{v^3}{9x}\big)$ is analytic for $x>0$ and $v\in\R$. It also turns out that for any $x\rightarrow0^+$ and $v\rightarrow0$ such that $x^{-1}{v^3}$ is bounded, we have $f(x,v)=O\big(x^\frac{1}{6}\big)$; more precisely, supposing $-\frac{v^3}{9x}\rightarrow c_0$, then 
\begin{align*}
f(x,v)=  x^\frac{1}{6}\Psi(c_0)+o\big(x^\frac{1}{6}\big). 
\end{align*}
Therefore, we conclude that the solution $f(x,v)$ to \eqref{steady} is not of the class $C^{\alpha_x,\alpha_v}_{x,v}$ near the boundary $\{x=0\}$ for any $\alpha_x>\frac{1}{6}$ or $\alpha_v>\frac{1}{2}$.


\begin{thebibliography}{10}
\bibitem{AM2021}
Dallas Albritton, Scott Armstrong, Jean-Christophe Mourrat, and Matthew Novack.
\newblock Variational methods for the kinetic {F}okker--{P}lanck equation.
\newblock {\em Anal. PDE}, 17(6):1953--2010, 2024.

\bibitem{BG}
Mohamed~Salah Baouendi and Pierre Grisvard.
\newblock Sur une \'equation d'\'evolution changeant de type.
\newblock {\em J. Functional Analysis}, 2:352--367, 1968.

\bibitem{Bardos}
Claude Bardos.
\newblock Probl\`emes aux limites pour les \'{e}quations aux d\'{e}riv\'{e}es partielles du premier ordre \`a coefficients r\'{e}els; th\'{e}or\`emes d'approximation; application \`a l'\'{e}quation de transport.
\newblock {\em Ann. Sci. \'{E}cole Norm. Sup. (4)}, 3:185--233, 1970.

\bibitem{BLD}
Adel Blouza and Herv\'{e} Le~Dret.
\newblock An up-to-the-boundary version of {F}riedrichs's lemma and applications to the linear {K}oiter shell model.
\newblock {\em SIAM J. Math. Anal.}, 33(4):877--895, 2001.

\bibitem{Boyer}
Franck Boyer.
\newblock Trace theorems and spatial continuity properties for the solutions of the transport equation.
\newblock {\em Differential Integral Equations}, 18(8):891--934, 2005.

\bibitem{Carrillo}
Jos\'e~A. {Carrillo}.
\newblock {Global weak solutions for the initial-boundary-value problems to the Vlasov-Poisson-Fokker-Planck system}.
\newblock {\em {Math. Methods Appl. Sci.}}, 21(10):907--938, 1998.

\bibitem{Cerci2}
Carlo Cercignani, Reinhard Illner, and Mario Pulvirenti.
\newblock {\em The mathematical theory of dilute gases}, volume 106 of {\em	Applied Mathematical Sciences}.
\newblock Springer-Verlag, New York, 1994.

\bibitem{Cessenat2}
Michel Cessenat.
\newblock Th\'{e}or\`emes de trace pour des espaces de fonctions de la neutronique.
\newblock {\em C. R. Acad. Sci. Paris S\'{e}r. I Math.}, 300(3):89--92, 1985.

\bibitem{Chand}
Subrahmanyan Chandrasekhar.
\newblock Stochastic problems in physics and astronomy.
\newblock {\em Rev. Modern Phys.}, 15:1--89, 1943.

\bibitem{DK}
Bj\"{o}rn E.~J. Dahlberg and Carlos~E. Kenig.
\newblock Hardy spaces and the {N}eumann problem in {$L^p$} for {L}aplace's equation in {L}ipschitz domains.
\newblock {\em Ann. of Math. (2)}, 125(3):437--465, 1987.

\bibitem{DG}
Ennio De~Giorgi.
\newblock Sulla differenziabilit\`a e l'analiticit\`a delle estremali degli integrali multipli regolari.
\newblock {\em Mem. Accad. Sci. Torino. Cl. Sci. Fis. Mat. Nat. (3)}, 3:25--43, 1957.
	
\bibitem{DPLode}
Ronald~J. DiPerna and Pierre-Louis Lions.
\newblock Ordinary differential equations, transport theory and {S}obolev spaces.
\newblock {\em Invent. Math.}, 98(3):511--547, 1989.

\bibitem{DGY}
Hongjie Dong, Yan Guo, and Timur Yastrzhembskiy.
\newblock Kinetic {F}okker-{P}lanck and {L}andau equations with specular reflection boundary condition.
\newblock {\em Kinetic \& Related Models}, 15(3):467--516, 2022.

\bibitem{Fichera1}
Gaetano Fichera.
\newblock Sulle equazioni differenziali lineari ellittico-paraboliche del secondo ordine.
\newblock {\em Atti Accad. Naz. Lincei Mem. Cl. Sci. Fis. Mat. Natur. Sez. Ia (8)}, 5:1--30, 1956.

\bibitem{GT}
David Gilbarg and Neil~S. Trudinger.
\newblock {\em Elliptic partial differential equations of second order}.
\newblock Classics in Mathematics. Springer-Verlag, Berlin, 2001.
\newblock Reprint of the 1998 edition.

\bibitem{GIMV}
Fran\c{c}ois Golse, Cyril Imbert, Cl\'{e}ment Mouhot, and Alexis~F. Vasseur.
\newblock Harnack inequality for kinetic {F}okker-{P}lanck equations with rough coefficients and application to the {L}andau equation.
\newblock {\em Ann. Sc. Norm. Super. Pisa Cl. Sci. (5)}, 19(1):253--295, 2019.

\bibitem{GHJO}
Yan {Guo}, Hyung~Ju {Hwang}, Jin~Woo {Jang}, and Zhimeng {Ouyang}.
\newblock {The Landau equation with the specular reflection boundary condition}.
\newblock {\em {Arch. Ration. Mech. Anal.}}, 236(3):1389--1454, 2020.

\bibitem{GKTT}
Yan {Guo}, Chanwoo {Kim}, Daniela {Tonon}, and Ariane {Trescases}.
\newblock {Regularity of the Boltzmann equation in convex domains}.
\newblock {\em {Invent. Math.}}, 207(1):115--290, 2017.

\bibitem{Ham}
Kamel Hamdache.
\newblock Initial-boundary value problems for the {B}oltzmann equation: global existence of weak solutions.
\newblock {\em Arch. Rational Mech. Anal.}, 119(4):309--353, 1992.

\bibitem{HS}
Christopher Henderson and Stanley Snelson.
\newblock {$C^\infty$} smoothing for weak solutions of the inhomogeneous {L}andau equation.
\newblock {\em Arch. Ration. Mech. Anal.}, 236(1):113--143, 2020.
	
\bibitem{Ho}
Lars H\"{o}rmander.
\newblock Hypoelliptic second order differential equations.
\newblock {\em Acta Math.}, 119:147--171, 1967.

\bibitem{HJV}
Hyung~Ju Hwang, Juhi Jang, and Juan J.~L. Vel\'{a}zquez.
\newblock The {F}okker-{P}lanck equation with absorbing boundary conditions.
\newblock {\em Arch. Ration. Mech. Anal.}, 214(1):183--233, 2014.

\bibitem{Jerison}
David~S. {Jerison}.
\newblock {The Dirichlet problem for the Kohn Laplacian on the Heisenberg group. II}.
\newblock {\em {J. Funct. Anal.}}, 43:224--257, 1981.

\bibitem{Keld}
Mstislav~V. Keldy\v{s}.
\newblock On certain cases of degeneration of equations of elliptic type on the boundry of a domain.
\newblock {\em Doklady Akad. Nauk SSSR (N.S.)}, 77:181--183, 1951.

\bibitem{Kim}
Chanwoo {Kim}.
\newblock {Formation and propagation of discontinuity for Boltzmann equation in non-convex domains}.
\newblock {\em {Commun. Math. Phys.}}, 308(3):641--701, 2011.

\bibitem{KN1}
Joseph~J. Kohn and Louis Nirenberg.
\newblock Non-coercive boundary value problems.
\newblock {\em Comm. Pure Appl. Math.}, 18:443--492, 1965.

\bibitem{Kol}
Andreĭ~N. Kolmogoroff.
\newblock Zuf\"{a}llige {B}ewegungen (zur {T}heorie der {B}rownschen {B}ewegung).
\newblock {\em Ann. of Math. (2)}, 35(1):116--117, 1934.

\bibitem{Lady}
Olga~A. Ladyzhenskaia, Vsevolod~A. Solonnikov, and Nina~N. Ural'tseva.
\newblock {\em {Linear and quasi-linear equations of parabolic type. Translated from the Russian by S. Smith}}, volume~23.
\newblock American Mathematical Society, Providence, RI, 1968.

\bibitem{Landau}
Lev~D. Landau.
\newblock Die kinetische {Gleichung} f{\"u}r den {Fall} {Coulombscher} {Wechselwirkung}.
\newblock {\em Phys. Z. Sowjetunion}, 10:154--164, 1936.

\bibitem{Maxwell}
James~Clerk Maxwell.
\newblock On stresses in rarified gases arising from inequalities of temperature.
\newblock {\em Philosophical Transactions of the royal society of London}, (170):231--256, 1879.

\bibitem{Mischler1}
St\'{e}phane Mischler.
\newblock On the trace problem for solutions of the {V}lasov equation.
\newblock {\em Comm. Partial Differential Equations}, 25(7-8):1415--1443, 2000.

\bibitem{Mischler3}
St\'ephane {Mischler}.
\newblock Kinetic equations with {M}axwell boundary conditions.
\newblock {\em {Ann. Sci. \'Ec. Norm. Sup\'er. (4)}}, 43(5):719--760, 2010.

\bibitem{Nier}
Francis Nier.
\newblock Boundary conditions and subelliptic estimates for geometric {K}ramers-{F}okker-{P}lanck operators on manifolds with boundaries. 
\newblock {\em Mem. Amer. Math. Soc.}, 252(1200):v+144, 2018.

\bibitem{Oleinik2}
Olga~A. Ole\u{\i}nik and Evgeniy~V. Radkevi\v{c}.
\newblock {\em Second order equations with nonnegative characteristic form}.
\newblock Plenum Press, New York-London, 1973.
\newblock Translated from the Russian by Paul C. Fife.

\bibitem{Tro}
Giovanni~Maria Troianiello.
\newblock {\em Elliptic differential equations and obstacle problems}.
\newblock The University Series in Mathematics. Plenum Press, New York, 1987.

\bibitem{Wangzhuxi}
Zhu-Xi Wang and Dun-Ren Guo.
\newblock {\em Special functions}.
\newblock World Scientific Publishing Co., Inc., Teaneck, NJ, 1989.
\newblock Translated from the Chinese by Dun-Ren Guo and Xue-Jiang Xia.

\bibitem{Zhang}
Liqun Zhang.
\newblock The {$C^\alpha$} regularity of a class of ultraparabolic equations.
\newblock {\em Commun. Contemp. Math.}, 13(3):375--387, 2011.

\bibitem{YZ}
Yuzhe {Zhu}.
\newblock {Velocity averaging and H\"older regularity for kinetic Fokker-Planck	equations with general transport operators and rough coefficients}.
\newblock {\em {SIAM J. Math. Anal.}}, 53(3):2746--2775, 2021.
\end{thebibliography}
\end{document}